\documentclass{amsart}
\usepackage{amssymb}
\usepackage{amsfonts}

\newtheorem{theorem}{Theorem}
\theoremstyle{plain}

\newtheorem{corollary}{Corollary}

\newtheorem{definition}{Definition}

\newtheorem{lemma}{Lemma}
\newtheorem{notation}{Notation}

\newtheorem{proposition}{Proposition}
\newtheorem{remark}{Remark}

\numberwithin{equation}{section}

\begin{document}
\title[Absolute continuity of orbital measures]{The absolute continuity of
convolutions of orbital measures in symmetric spaces}
\author{Sanjiv Kumar Gupta}
\address{Dept. of Mathematics and Statistics\\
Sultan Qaboos University\\
P.O.Box 36 Al Khodh 123\\
Sultanate of Oman}
\email{gupta@squ.edu.om}
\author{Kathryn E. Hare}
\address{Dept. of Pure Mathematics\\
University of Waterloo\\
Waterloo, Ont.,~Canada\\
N2L 3G1}
\email{kehare@uwaterloo.ca}
\thanks{This research is supported in part by NSERC \#44597. The first
author thanks the University of Waterloo for their hospitality when some of
this research was done.}

\begin{abstract}
We characterize the absolute continuity of convolution products of orbital
measures on the classical, irreducible Riemannian symmetric spaces $G/K$ of
Cartan type $III$, where $G$ is a non-compact, connected Lie group and $K$
is a compact, connected subgroup. By the orbital measures, we mean the
uniform measures supported on the double cosets, $KzK,$ in $G$. The
characterization can be expressed in terms of dimensions of eigenspaces or
combinatorial properties of the annihilating roots of the elements $z$.

A consequence of our work is to show that the convolution product of any rank%
$G/K,$ continuous, $K$-bi-invariant measures is absolutely continuous in any
of these symmetric spaces, other than those whose restricted root system is
type $A_{n}$ or $D_{3}$, when rank$G/K$ $+1$ is needed.
\end{abstract}

\maketitle

\section{Introduction}

In this paper, we study the smoothness properties of $K$-bi-invariant
measures on the irreducible Riemannian symmetric spaces $G/K$, where $G$ is
a connected Lie group and $K$ is a compact, connected subgroup fixed by a
Cartan involution of $G$. Inspired by earlier work of Dunkl \cite{Du} on
zonal measures on spheres, Ragozin in \cite{Ra} and \cite{Ra2} showed that
the convolution of any $\dim G/K,$ continuous, $K$-bi-invariant measures on $%
G$ was absolutely continuous with respect to the Haar measure on $G$. This
was improved by Graczyk and Sawyer, who in \cite{GSJFA} showed that if $G$
was non-compact and $n=$ rank$G/K$, then any $n+1$ convolutions of such
measures is absolutely continuous and that this is sharp for the symmetric
spaces whose restricted root system was type $A_{n}$. They conjectured that $%
n+1$ was always sharp. One consequence of our work is to show this
conjecture is false. In fact, for all the classical, non-compact symmetric
spaces of rank $n$, other than those whose restricted root system is type $%
A_{n}$, the convolution product of any $n$ $K$-bi-invariant, continuous
measures is absolutely continuous and this is sharp.

We obtain this result by studying a particular class of examples of $K$%
-bi-invariant, continuous measures, the so-called orbital measures $\nu
_{z}=m_{K}\ast \delta _{z}\ast m_{K},$ where $m_{K}$ is the Haar measure on $%
K$. These are the uniform measures supported on the double cosets $KzK$ in $%
G $. They are purely singular, probability measures. The main objective of
this paper is to characterize the $L$-tuples $(z_{1},...,z_{L})$ such that $%
\nu _{z_{1}}\ast \cdot \cdot \cdot \ast \nu _{z_{L}}$ is absolutely
continuous for the classical symmetric spaces of Cartan type $III$.

The convolution product $\nu _{z_{1}}\ast \cdot \cdot \cdot \ast \nu
_{z_{L}} $ is supported on the product of double cosets $Kz_{1}Kz_{2}\cdot
\cdot \cdot Kz_{L}K$ and hence if the convolution product is absolutely
continuous than the product of double cosets has postive Haar measure in $G$%
. In fact, if the convolution product is absolutely continuous, the product
of double cosets has non-empty interior and the converse is true, as well,
thus we also characterize which products of double cosets have non-empty
interior.

Given any $z_{j}\in G$ there is some $Z_{j}\in \mathfrak{g}$, the Lie
algebra of $G$, such that $z_{j}=\exp Z_{j}$. Our characterization is in
terms of combinatorial properties of the set of annihilating roots of the
elements $Z_{j}$. It also can be expressed in terms of the dimensions of the
largest eigenspaces when we view the $Z_{j}$ as matrices in the classical
Lie algebras.

In a series of papers, (see \cite{GSLie} - \cite{GSArxiv} and the examples
cited therein), Graczyk and Sawyer found a characterization for the absolute
continuity of $\nu _{x}\ast \nu _{y}$ for certain of the type $III$ (mainly)
classical symmetric spaces. Our approach was inspired by their work, but is
more abstract and relies heavily upon combinatorial properties of the root
systems and root spaces of Lie algebras.

The Cartan involution also gives rise to a decomposition of the Lie algebra
as $\mathfrak{g=k\oplus p}$. Take a maximal abelian subspace $\mathfrak{a}$
of $\mathfrak{p}$ and put $A=\exp \mathfrak{a}$. Then $G=KAK$, thus we can
always assume $z\in A$ when studying orbital measures. Closely related to
the $K$-bi-invariant orbital measures supported on double cosets in $G$ are
the $K$-invariant, uniform measures, $\mu _{Z},$ supported on the orbits
under the $Ad(K)$ action of elements $Z\in \mathfrak{a}$. It is known (\cite%
{AG}) that $\mu _{Z_{1}}\ast \cdot \cdot \cdot \ast \mu _{Z_{L}}$ is
absolutely continuous with respect to Lebesgue measure on $\mathfrak{p}$ if
and only if $\nu _{z_{1}}\ast \cdot \cdot \cdot \ast \nu _{z_{L}}$ is
absolutely continuous on $G$ when $z_{j}=\exp Z_{j},$ and this is also
equivalent to the sum of the $Ad(K)$ orbits generated by the $Z_{j}$ having
non-empty interior.

The absolute continuity of convolution products of orbital measures is
connected with questions about spherical functions $\phi _{\lambda }(\exp X)$
where $\lambda $ is a complex-valued linear form on $\mathfrak{a}$ and $X$ $%
\in \mathfrak{a}$. This is the spherical Fourier transform of the orbital
measure $\nu _{\exp X}$. The product formula states that 
\begin{equation*}
\phi _{\lambda }(e^{Z_{1}})\cdot \cdot \cdot \phi _{\lambda
}(e^{Z_{L}})=\int_{G}\phi _{\lambda }d\nu _{z_{1}}\ast \cdot \cdot \cdot
\ast \nu _{z_{L}},
\end{equation*}%
hence the absolute continuity of convolutions of orbital measures gives a
formula for a product of spherical functions.

An example of a compact, symmetric space is $(G\times G)/\Delta (G)$ where $G
$ is a compact, connected, simple Lie group and $\Delta (G)=\{(g,g):g\in
G\}\simeq G$. These are the symmetric spaces of Cartan type $II$. In this
case, the orbital measure $\nu _{z}$ for $z=(g,g),$ supported on the double
coset $\Delta (G)z\Delta (G)$, can be identified with the uniform measure
supported on the conjugacy class in $G$ containing the element $g$. The $%
\Delta (G)$-invariant measure, $\mu _{Z}$ for $Z\in \mathfrak{g},$ can be
identified with the measure on the compact Lie algebra $\mathfrak{g}$ that
is uniformly distributed on the adjoint orbit in $\mathfrak{g}$ containing $Z
$. These symmetric spaces are dual to the non-compact, symmetric spaces of
Cartan type $IV$, $G^{\mathbb{C}}/G$, where $G^{\mathbb{C}}$ is the
non-compact Lie group corresponding to the Lie algebra $\mathfrak{g\oplus ig}
$ (over $\mathbb{R}$) (see \cite{AG}, \cite{He}). It can be seen from \cite%
{AG} that the absolute continuity problem for Cartan type $IV$ symmetric
spaces can be deduced from the analogous problem for the corresponding $G$%
-invariant orbital measures $\mu _{Z},$ for $Z\in \mathfrak{g,}$ from the
(dual) Cartan type $II$ spaces.

In a series of papers, the authors (with various coauthors) studied the
absolute continuity problem for orbital measures in the compact setting. The
sharp exponent $n=n(Z)$ (or $n(z)$) with the property that the $n$-fold
convolution product of $\mu _{Z}$ (or $\nu _{z})$ was absolutely continuous
was determined for the classical compact Lie algebras $\mathfrak{g}$ in \cite%
{GHMathZ}, for the classical compact Lie groups in \cite{GHAdv}, and for the
exceptional compact Lie groups and algebras in \cite{HJSY}. Sufficient
conditions for these problems were found using harmonic analysis methods not
generally available in the symmetric space setting.

In \cite{Wr}, Wright used geometric methods to extend this result to the
convolution of different orbital measures on the Lie algebras of type $A_{n}$%
. Later, the authors in \cite{GHAbs} obtained a (almost complete)
characterization for the absolute continuity of convolution products of
arbitrary orbital measures in all the classical compact Lie algebras and
hence also those for the classical symmetric spaces of Cartan type $IV$.
This was done by mainly algebraic/combinatorial methods. These ideas are key
to this paper, particularly for the symmetric spaces with one-dimensional,
restricted root spaces, but many additional technical complications arise in
the more general symmetric space setting.

Earlier, Ricci and Stein in \cite{RS1} and \cite{RS2} studied the smoothness
properties of convolutions of measures supported on manifolds whose product
has non-empty interior. They proved, for example, that if the surface
measure of a compact manifold has an absolutely continuous convolution
product, then the density function of that convolution product is actually
in $L^{1+\varepsilon }$ for some $\varepsilon >0$. A number of authors have
attempted to compute the density function (in some special cases), but this
is very hard. We refer the reader to \cite{DRW}, \cite{FG} and \cite{GS2004}%
, for example. Sums of adjoint orbits have also been studied, such as in 
\cite{KT} where the sum of two adjoint orbits in $su(n)$ is described. The
smoothing properties of convolution is also of interest in the study of
random walks on groups and hypergroups; c.f., \cite{BH}, \cite{JR}, \cite{LW}%
.

\subsection{Organization of the paper}

We begin in the second section by introducing terminology, including the
definition of orbital measures and the very important notion of annihilating
roots. We also explain the connection between the absolute continuity
problem for the two classes of orbital measures and the connection with
questions about sums of orbits / products of double cosets. In section three
we explain what is meant by the type of an element, and what is meant by
eligible and exceptional tuples. These ideas come from \cite{GHAbs}. We also
give the formal statement of our main theorem, that absolute continuity is
characterized by eligibility and non-exceptionality. An immediate corollary
is that except when the restricted root space is type $A_{n}$, any
convolution product of $n=$ rank$G/K$ continuous bi-invariant measures on $G$
is absolutely continuous. Moreover, we can describe which $(n-1)$-fold
products of orbital measures are not absolutely continuous. The proof of
sufficiency of our characterization is the content of section four and
occupies the majority of the paper. Finally, in section five we prove that
the non-eligible and the exceptional tuples do not give rise to absolutely
continuous convolution products. Useful basic facts about the symmetric
spaces of Cartan type $III$ are summarized in the appendix.

\section{Set Up and Preliminary results}

\subsection{Cartan decomposition}

Let $G$ be a non-compact, connected, Lie group with Lie algebra $\mathfrak{g}
$ and suppose $\theta $ is an involution of $G$. Let $K=\{g\in G:\theta
(g)=g\}$ and assume $K$ is compact and connected. The quotient space, $G/K$,
is called a symmetric space. We let $W$ denote its Weyl group.

The map $\theta $ induces an involution of $\mathfrak{g}$, also denoted $%
\theta $. We put 
\begin{equation*}
\mathfrak{k}=\left\{ X\in \mathfrak{g\mid \theta }\left( X\right) =X\right\} 
\text{ and }\mathfrak{p}=\left\{ X\in \mathfrak{g\mid \theta }\left(
X\right) =-X\right\} ,
\end{equation*}%
the $\pm 1$ eigenspaces of $\theta $, respectively. The decomposition $%
\mathfrak{g}=\mathfrak{k\oplus p}$ is called the Cartan decomposition of the
Lie algebra $\mathfrak{g}.$ The subspaces $\mathfrak{k}$ and $\mathfrak{p}$
satisfy the following rules: 
\begin{equation*}
\left[ \mathfrak{k},\ \mathfrak{k}\right] \subseteq \mathfrak{k,}\ \left[ 
\mathfrak{k},\ \mathfrak{p}\right] \subseteq \mathfrak{p,}\ \left[ \mathfrak{%
p},\ \mathfrak{p}\right] \subseteq \mathfrak{k.}
\end{equation*}%
We fix a maximal abelian (as a subalgebra of $\mathfrak{g}$) subspace $%
\mathfrak{a}$ of $\mathfrak{p}$ and let $\mathfrak{a}^{\ast }$ be the dual
of $\mathfrak{a}$.

As is standard, we will let $ad(X)(Y)=[X,Y],$ denote by $Ad(\cdot )$ the
adjoint action of $G$ on $\mathfrak{g}$ and write $\exp $ for the
exponential map from $\mathfrak{g}$ to $G$.

It is known that $Ad(k):\mathfrak{p}\rightarrow \mathfrak{p}$ whenever $k\in
K$. Moreover, $\exp \mathfrak{k}=K$ and if we put $A=\exp \mathfrak{a}$,
then $G=KAK$ (\cite[p. 459]{Kn}).

By $\mathfrak{a}^{+}$ we mean the subset 
\begin{equation*}
\mathfrak{a}^{+}=\{H\in \mathfrak{a:\alpha }(H)>0\text{ for all }\alpha 
\mathfrak{\in a}^{\ast }\}.
\end{equation*}%
The sets $w(\mathfrak{a}^{+})$ are disjoint for distinct $w\in W$ and $%
\mathfrak{a=}\bigcup_{w\in W}w(\overline{\mathfrak{a}^{+}}).$

Put $A^{+}=\exp \mathfrak{a}^{+}$. It is known \textbf{(}\cite[p.402]{He}%
\textbf{)} that $G=K\overline{A^{+}}K$. Indeed, given any $g\in G$, there is
a pair $k_{1},k_{2}\in K$ and a unique $H\in \overline{\mathfrak{a}^{+}}$
such that $g=k_{1}(\exp H)k_{2}$. We define a map $\mathcal{A}:G\rightarrow 
\overline{\mathfrak{a}^{+}}\subseteq \mathfrak{a}$ by $\mathcal{A}(g)=H.$ We
also speak of $\mathcal{A}$ as a map from $G\rightarrow \overline{A^{+}}$ 
\textbf{\ }by taking $\mathcal{A}(g)=\exp H$. It will be clear from the
context which we mean.

For non-zero $\alpha \in \mathfrak{a}^{\ast }$ we consider the set 
\begin{equation*}
\mathfrak{g}_{\alpha }=\left\{ X\in \mathfrak{g:}\left[ H,X\right] =\alpha
\left( H\right) X\text{ for all }H\in \mathfrak{a}\right\} .
\end{equation*}%
The set of restricted roots, $\Phi $, is defined by 
\begin{equation*}
\Phi =\left\{ \alpha \in \mathfrak{a}^{\ast }:\mathfrak{g}_{\alpha }\neq
0\right\}
\end{equation*}%
and the subset of positive roots is denoted by $\Phi ^{+}$. The set $\Phi $
is a root system, although not necessarily reduced because it is possible
for both $\alpha $ and $2\alpha $ to be a root.

The vector spaces $\mathfrak{g}_{\alpha }$ corresponding to $\alpha \in \Phi 
$ are known as the restricted root spaces and need not be one-dimensional.
It is well known that $\theta \mathfrak{g}_{\alpha }=\mathfrak{g}_{-\alpha }$%
. The Lie algebra $\mathfrak{g}$ can be decomposed as 
\begin{equation*}
\mathfrak{g}\mathfrak{=}\mathfrak{g}_{0}\oplus \sum\limits_{\alpha \in \Phi }%
\mathfrak{g}_{\alpha },\text{ where }\mathfrak{g}_{0}=\mathfrak{a}\oplus 
\mathfrak{m}
\end{equation*}%
with $\mathfrak{m}=\left\{ X\in \mathfrak{k:}\left[ X,\mathfrak{a}\right]
=0\right\} $.

The $\pm 1$ eigenspaces of $\theta $ can also be described as 
\begin{equation*}
\mathfrak{k}={sp}\left\{ X+\theta X:X\in \mathfrak{g}_{\alpha },\ \alpha \in
\Phi ^{+}\bigcup \{0\}\right\} ={sp}\left\{ X+\theta X:X\in \mathfrak{g}%
_{\alpha },\ \alpha \in \Phi ^{+}\right\} \oplus \mathfrak{m}
\end{equation*}%
and 
\begin{equation*}
\mathfrak{p}={sp}\left\{ X-\theta X:X\in \mathfrak{g}_{\alpha },\ \alpha \in
\Phi ^{+}\right\} \oplus \mathfrak{a}
\end{equation*}%
where by $sp$ we will mean the real span.

We will put 
\begin{equation*}
X^{+}=X+\theta X\text{ and }X^{-}=X-\theta X.
\end{equation*}

We will write $\mathfrak{g}_{n}$, $\mathfrak{k}_{n}$ etc if we want to
emphasize that the rank of the symmetric space is $n$.

Throughout this paper we will assume $G/K$ is an irreducible, Riemannian,
globally symmetric space of Type III for which the root system of the Lie
group $G$ is not exceptional. These have Cartan classifications $AI,AII,AIII$%
, $BI,$ $CI,CII$, $DI$ and $DIII$. Their restricted root systems have Lie
types $A_{n}$, $B_{n}$, $C_{n}$, $D_{n}$ or $BC_{n}$. In the appendix, we
summarize basic information about these symmetric spaces, including the
choices of $G,K$, rank and the dimension of the restricted root spaces. This
information is taken from \cite{Bu}, \cite{CM}, \cite{He} and \cite{Kn}. For
general facts about root systems, we refer the reader to \cite{Hu} and \cite%
{Ka}.

\subsection{Orbital measures, orbits and double cosets}

A measure $\mu $ on $\mathfrak{p}$ will be said to be $K$-invariant if $\mu
(E)=\mu (Ad(k)E)$ for all $k\in K$ and Borel sets $E\subseteq \mathfrak{p}$.
Corresponding to each $Z\in \mathfrak{p}$ is a $K$-invariant probability
measure called the\textit{\ orbital measure,} $\mu _{Z}$, defined by%
\begin{equation*}
\int_{\mathfrak{p}}fd\mu _{Z}=\int_{K}f(Ad(k)Z)dm_{K}(k)
\end{equation*}%
for any continuous, compactly supported function $f$ on $\mathfrak{p}$. Here 
$m_{K}$ denotes the Haar measure on the compact group $K$. This measure is
supported on the $Ad_{K}$-orbit of $Z$, meaning the orbit of $Z$ under the
action of $K$ on $\mathfrak{p}$. We denote this set by $O_{Z}$, thus

\begin{equation*}
O_{Z}=\{Ad(k)Z:k\in K\}\text{.}
\end{equation*}%
Every $Ad_{K}$-orbit contains an element of $\mathfrak{a}$ since $\mathfrak{p%
}=\bigcup_{k\in K}Ad(k)\mathfrak{a}$ (\cite[p. 455]{Kn}), hence in studying
orbital measures, $\mu _{Z}$, there is no loss in assuming $Z\in \mathfrak{a}
$.

$Ad_{K}$-orbits are manifolds of proper dimension in $\mathfrak{p}$ and
hence have Lebesgue measure zero and empty interior. Thus the orbital
measures are singular with respect to Lebesgue measure.

A measure $\mu $ on $G$ will be said to be $K$-bi-invariant if $\mu (E)=\mu
(k_{1}Ek_{2})$ for all $k_{1},k_{2}\in K$ and Borel sets $E\subseteq G$.
These measures can be naturally identified with the left $K$-invariant
measures on the symmetric space $G/K$. An example is the measure we will
denote by $\nu _{z}$ for $z\in G,$ defined by 
\begin{equation*}
\int_{G}f\,d\nu _{z}=\int_{K}\int_{K}f(k_{1}zk_{2})dm_{K}(k_{1})dm_{K}(k_{2})
\end{equation*}%
for any compactly supported, continuous function $f$ on $G$. This is the
measure supported on the double coset $KzK$ in $G$ and is also called an
orbital measure. It is easy to see that $\nu _{z}=m_{K}\ast \delta _{z}\ast
m_{K}$ where $\delta _{z}$ is the point mass measure at $z$. Since $G=KAK$,
every double coset contains an element of $A$ and hence there is no loss of
generality in assuming $z\in A$.

Double cosets are also manifolds of proper dimension, hence have Haar
measure zero in $G$ and empty interior. It follows that the measures $\nu
_{z}$ are singular with respect to Haar measure.

The $K$-bi-invariant measures are also known as zonal measures.

\subsection{Annihilating roots and Tangent spaces}

Given $Z\in \mathfrak{a}$, we let 
\begin{equation*}
\Phi _{Z}=\{\alpha \in \Phi :\alpha (Z)=0\}
\end{equation*}%
be the set of annihilating roots of $Z$ and let 
\begin{equation*}
\mathcal{N}_{Z}=sp\{X_{\alpha }-\theta X_{\alpha }:X_{\alpha }\in \mathfrak{g%
}_{\alpha }\text{, }\alpha \notin \Phi _{Z}\}\subseteq \mathfrak{p}\text{.}
\end{equation*}%
The set $\Phi _{Z}$ is itself a root system and is a proper root subsystem
provided $Z\neq 0$. As we will see, these root subsystems, $\Phi _{Z},$ and
the associated spaces, $\mathcal{N}_{Z},$ are of fundamental importance in
studying orbits and orbital measures.

A very useful fact is that if $Z\in \mathfrak{a}$ and $X_{\alpha }\in 
\mathfrak{g}_{\alpha }$, then 
\begin{equation}
\lbrack Z,X_{\alpha }^{+}]=[Z,X_{\alpha }+\theta X_{\alpha }]=\alpha
(Z)(X_{\alpha }-\theta X_{\alpha })=\alpha (Z)X_{\alpha }^{-}.  \label{adZ}
\end{equation}%
In particular, 
\begin{equation*}
\mathcal{N}_{Z}=sp\{[Z,X]:X\in \mathfrak{k}\}=\text{Im}ad(Z)|_{\mathfrak{k}}%
\text{.}
\end{equation*}

It is well known that the tangent space to the $Ad_{K}$-orbit of $Z$ is $%
T_{Z}(O_{Z})=\{[Z,Y]:Y\in \mathfrak{k}\}.$ Hence%
\begin{equation*}
T_{Z}(O_{Z})=\mathcal{N}_{Z}
\end{equation*}%
and, in particular,%
\begin{equation*}
\dim O_{Z}=\dim \mathcal{N}_{Z}=\sum_{\alpha \in \Phi ^{+}\diagdown \Phi
_{Z}}\dim \mathfrak{g}_{\alpha }.
\end{equation*}%
More generally, if $X\in O_{Z}$, say $X=Ad(k)Z$, then $%
T_{X}(O_{Z})=Ad(k)T_{Z}(O_{Z})=Ad(k)\mathcal{N}_{Z}$.

Ragozin \cite{Ra} proved that questions about the absolute continuity of
convolution products of orbital measures are related to geometric questions
about orbits and tangent spaces. Here are some key ideas.

\begin{proposition}
\label{keyprop} Let $Z_{1},...,Z_{t}\in \mathfrak{a}$ and $z_{j}=\exp
Z_{j}\in A$. The following are equivalent:

\begin{enumerate}
\item The convolution product of orbital measures $\mu _{Z_{1}}\ast \cdot
\cdot \cdot \ast \mu _{Z_{t}}$ is absolutely continuous with respect to
Lebesgue measure on $\mathfrak{p}$.

\item The sum $O_{Z_{1}}+\cdot \cdot \cdot +O_{Z_{t}}$ has non-empty
interior (equivalently, positive Lebesgue measure) in $\mathfrak{p}$.

\item There is some $k_{1}=Id$, $k_{2},...,k_{t}\in K$ such that 
\begin{equation*}
sp\{Ad(k_{j})\mathcal{N}_{Z_{j}}:j=1,...,t\}=\mathfrak{p.}
\end{equation*}

\item The convolution product of orbital measures $\nu _{z_{1}}\ast \cdot
\cdot \cdot \ast \nu _{z_{t}}$ is absolutely continuous with respect to Haar
measure on $G$.

\item The product of the double cosets $Kz_{1}Kz_{2}K\cdot \cdot \cdot
Kz_{t}K$ has non-empty interior (equivalently, positive Haar measure) in $G$.

\item There is some $k_{2},...,k_{t}\in K$ such that 
\begin{equation}
\{X_{1}+Ad(z_{1})X_{2}+\cdot \cdot \cdot +Ad(z_{1}k_{2}z_{2}\cdot \cdot
\cdot k_{t}z_{t})X_{t+1}:X_{j}\in \mathfrak{k}\}=\mathfrak{g.}
\label{diffrange}
\end{equation}
\end{enumerate}

Furthermore, in the case that (3) or (6) holds for some $(t-1)$-tuple, $%
(k_{2},...,k_{t})$, then it holds for almost all $(k_{2},...,k_{t})\in
K^{t-1}$.
\end{proposition}

\begin{proof}
This is an amalgamation of ideas that can mainly be found in \cite{AG} and 
\cite{Ra}. Consider the maps%
\begin{equation*}
F=F_{Z_{1},...,Z_{t}}:O_{Z_{1}}\times \cdot \cdot \cdot \times
O_{Z_{t}}\rightarrow \mathfrak{p}
\end{equation*}%
\begin{equation*}
F(X_{1},...,X_{t})=X_{1}+\cdot \cdot \cdot +X_{t}
\end{equation*}%
and 
\begin{equation*}
f=f_{z_{1},...,z_{t}}:K^{t+1}\rightarrow G
\end{equation*}%
\begin{equation*}
f(k_{1},...,k_{t+1})=k_{1}z_{1}k_{2}\cdot \cdot \cdot k_{t}z_{t}k_{t+1}.
\end{equation*}

We remark that if the rank of $f$ is equal to the dimension of $G$ at one
point, then by an analyticity argument it is equal to $\dim G$ at almost
every point. Ragozin proves that in this case $\nu _{z_{1}}\ast \cdot \cdot
\cdot \ast \nu _{z_{t}}$ is absolutely continuous with respect to Haar
measure, $m_{G},$ on $G$ and that the image of $f,$ the product $%
Kz_{1}Kz_{2}K\cdot \cdot \cdot Kz_{t}K,$ has non-empty interior. But the
range of the differential of $f_{z_{1},...,z_{t}}$ at $(k_{1},...,k_{t+1})$
is the left hand side of (\ref{diffrange}) and hence $rankf=\dim G$ if and
only if (\ref{diffrange}) holds.

Conversely, if $rankf<\dim G$ at all points, then Sard's theorem implies the
measure of the image of $f$ is zero. Hence $m_{G}(Kz_{1}Kz_{2}K\cdot \cdot
\cdot Kz_{t}K)=0$ and this forces $\nu _{z_{1}}\ast \cdot \cdot \cdot \ast
\nu _{z_{t}}$ to be a singular measure. These arguments prove the
equivalence of (4)-(6).

The equivalence of (1)-(3) is similar upon noting that the range of the
differential of $F_{Z_{1},...,Z_{t}}$ at $(X_{1},...,X_{t})$, $X_{j}\in
O_{Zj},$ is $\sum_{j=1}^{t}T_{X_{j}}(O_{Z_{j}})$ and that if $%
X_{j}=Ad(k_{j})Z_{j}$, then $T_{X_{j}}(O_{Z_{j}})=Ad(k_{j})\mathcal{N}%
_{Z_{j}}$ (see \cite{GHJMAA}).

In the proof of Theorem 3.1 of \cite{AG} the authors prove that (\ref%
{diffrange}) holds in the special case that all $Z_{j}$ are equal if and only%
\begin{equation}
\mathfrak{k\oplus }sp\{Ad(k_{j})\mathcal{N}_{Z_{j}}:j=1,...,t\}=\mathfrak{g.}
\label{diff2}
\end{equation}%
But the same argument works for general $Z_{j}$. As $Ad(k_{j})\mathcal{N}%
_{Z_{j}}\subseteq \mathfrak{p}$ for any $k_{j}\in K$, (\ref{diff2}) holds if
and only if property (3) holds, i.e., $sp\{Ad(k_{j})\mathcal{N}%
_{Z_{j}}:j=1,...,t\}=\mathfrak{p.}$
\end{proof}

To show that a convolution product of orbital measures is absolutely
continuous, we will typically establish that property (3) of the proposition
holds.

\begin{notation}
We will call $(Z_{1},...,Z_{m})$ an absolutely continuous tuple if any of
these equivalent conditions are satisfied.
\end{notation}

For emphasis, we highlight:

\begin{corollary}
\label{equiv} $(Z_{1},...,Z_{m})$ is an absolutely continuous tuple if and
only if the orbital measure $\nu _{z_{1}}\ast \cdot \cdot \cdot \ast \nu
_{z_{m}}$ is absolutely continuous on $G$ for $z_{j}=\exp Z_{j}$.
\end{corollary}

\section{Statement of the Characterization Theorem}

\bigskip

\subsection{Type and eligibility}

As in \cite{GHAbs}, our theorem will depend upon what we call type and
eligibility. Here we modify those definitions for the symmetric space
scenario.

\subsubsection{Type of an element}

When the restricted root system of the symmetric space is of type $A_{n-1}$
(we also call this type $SU(n)$ as this is the classical Lie group whose
root system is type $A_{n-1}$), after applying a suitable Weyl conjugate any 
$Z$ $\in \mathfrak{a}_{n}$ can be identified with the $n$-vector 
\begin{equation*}
Z=(\underbrace{a_{1},\dots ,a_{1}}_{s_{1}},\underbrace{a_{2},\dots ,a_{2}}%
_{s_{2}},\dots ,\underbrace{a_{m},\dots ,a_{m}}_{s_{m}}),
\end{equation*}%
where the $a_{j}\in \mathbb{R}$ are distinct and $\sum_{j=1}^{m}s_{j}a_{j}=0$%
. The set of annihilating roots $\Phi _{Z}=$ $\Psi _{1}\cup \cdots \cup \Psi
_{m}$ where 
\begin{align*}
\Psi _{1}^{+}& =\{e_{i}-e_{j}:1\leq i<j\leq s_{1}\}\text{ and} \\
\Psi _{l}^{+}& =\{e_{i}-e_{j}:s_{1}+\cdots +s_{l-1}<i<j\leq s_{1}+\cdots
+s_{l}\}\text{ for }l>1.
\end{align*}%
Following \cite{GHMathZ}, we say that $Z$ is\textit{\ type} $SU(s_{1})\times
\cdots \times SU(s_{m})$ as this is the Lie type of its set of annihilating
roots.

If the restricted root system is type $B_{n}$, $C_{n}$, $D_{n}$ or $BC_{n}$,
then up to a Weyl conjugate, $Z\in \mathfrak{a}_{n}$ can be identified with
the $n$-vector 
\begin{equation*}
Z=(\underbrace{0,\dots ,0}_{J},\underbrace{a_{1},\dots ,a_{1}}_{s_{1}},\dots
,\underbrace{a_{m},\dots ,(\pm )a_{m}}_{s_{m}})
\end{equation*}%
where the $a_{j}>0$ are distinct. We remark that the minus sign is needed
only in type $D_{n}$ and only if $J=0$. (This is because the Weyl group in
type $D_{n}$ changes only an even number of signs.)

The set of annihilating roots of $Z$ can be written as $\Phi _{Z}=\Psi
_{0}\cup \Psi _{1}\cdots \cup \Psi _{m}$ where 
\begin{equation*}
\Psi _{0}^{+}=\left\{ 
\begin{array}{cc}
\{e_{k},e_{i}\pm e_{j}:1\leq i,j,k\leq J,i<j\}\text{ } & \text{if }\Phi 
\text{ type }B_{n} \\ 
\{2e_{k},e_{i}\pm e_{j}:1\leq i,j,k\leq J,i<j\}\text{ } & \text{if }\Phi 
\text{ type }C_{n} \\ 
\{e_{i}\pm e_{j}:1\leq i<j\leq J\} & \text{if }\Phi \text{ type }D_{n} \\ 
\{e_{k},2e_{k},e_{i}\pm e_{j}:1\leq i,j,k\leq J,i<j\} & \text{if }\Phi \text{
type }BC_{n}%
\end{array}%
\right.
\end{equation*}%
and for $l\geq 1$,%
\begin{equation*}
\Psi _{l}^{+}=\{e_{i}-e_{j}:J+s_{1}+\cdots +s_{l-1}<i<j\leq J+s_{1}+\cdots
+s_{l}\},
\end{equation*}%
except if $Z=(a_{1},\dots ,a_{1},\dots ,a_{m},\dots ,-a_{m})$ in $D_{n}$
when 
\begin{equation*}
\Psi _{m}^{+}=\{e_{i}-e_{j},e_{i}+e_{n}:n-s_{m}<i<j\leq n-1\}.
\end{equation*}

In the case that $\Phi $ is type $B_{n}$, we will say that such an element $%
Z $ is \textit{type} 
\begin{equation*}
B_{J}\,\times SU(s_{1})\times \cdots \times SU(s_{m})\text{ }
\end{equation*}%
as this is the Lie type of $\Phi _{Z}$. We make a similar definition if $%
\Phi $ is type $C_{n}$, $D_{n}$ or $BC_{n}$. We understand $SU(1)$ and $%
B_{0} $ to be empty, $B_{1}$ to be the subsystem $\{e_{1}\}$ and define $%
C_{0}$, $BC_{0}$, $C_{1}$ and $BC_{1}$ similarly. In the case of type $%
D_{n}, $ we understand both $D_{0}$ and $D_{1}$ to be empty and $D_{2}$ to
be $\{e_{i}\pm e_{j}\}$. We often omit the writing of the empty root systems
in our descriptions.

Note that there are two distinct subsystems (up to Weyl conjugacy) of
annihilating roots of elements of type $SU(n)$ in $D_{n}$.

\subsubsection{Dominant type}

Suppose the symmetric space has restricted root system of type $B_{n}$ and $%
Z\in a_{n}$ is type $B_{J}\,\times SU(s_{1})\times \cdots \times SU(s_{m})$.
We will say $Z$ is \textit{dominant }$B$\textit{\ type} if $2J\geq \max
s_{j} $, and is \textit{dominant }$SU$\textit{\ type} otherwise. We define%
\textit{\ dominant }$C$\textit{, }$D$\textit{\ and }$BC$\textit{\ type}
similarly for $Z$ in a symmetric space with restricted root system of type $%
C_{n},D_{n} $ or $BC_{n}$.

\subsubsection{Eligible and Exceptional Tuples}

\begin{notation}
If $Z$ is of type $SU(s_{1})\times \cdots \times SU(s_{m})$ in a symmetric
space with restricted root system of type $A_{n}$, put $S_{X}=\max s_{j}$.

If $Z$ is type $B_{J}\times SU(s_{1})\times \cdots \times SU(s_{m})$ in a
symmetric space with restricted root system of type $B_{n}$, put 
\begin{equation*}
S_{X}=%
\begin{cases}
2J & \text{ if }X\text{ is dominant }B\text{ type} \\ 
\max s_{j} & \text{else}%
\end{cases}%
.
\end{equation*}%
Define $S_{X}$ similarly when $Z$ is in a symmetric space with restricted
root system of type $\,C_{n},D_{n}$ or $BC_{n}$.
\end{notation}

\begin{definition}
(1) We will say that the $L$-tuple $(Z_{1},Z_{2},\dots ,Z_{L})\in \mathfrak{a%
}^{L}$ in a symmetric space with restricted root system of type $A_{n}$ is 
\textbf{eligible} if 
\begin{equation*}
\sum_{i=1}^{L}S_{X_{i}}\leq (L-1)(n+1).
\end{equation*}

(2) We will say that the $L$-tuple $(Z_{1},Z_{2},\dots ,Z_{L})$ $\in 
\mathfrak{a}^{L}$ in a symmetric space with restricted root system of type $%
B_{n}$, $C_{n},$ $D_{n}$ or $BC_{n}$ is \textbf{eligible} if 
\begin{equation}
\sum_{i=1}^{L}S_{X_{i}}\leq (L-1)2n.  \label{eligiblecriteria}
\end{equation}
\end{definition}

\begin{definition}
We will say that $(Z_{1},Z_{2},\dots ,Z_{L})\in \mathfrak{a}^{L}$ is an 
\textbf{exceptional tuple} in any of the following situations:

\begin{enumerate}
\item The symmetric space has restricted root system of type $A_{2n-1}$, $%
L=2 $, $n\geq 2$ and $Z_{1}$ and $Z_{2}$ are both of type $SU(n)\times SU(n)$%
;

\item The symmetric space has restricted root system of type $D_{n}$, $L=2$, 
$Z_{1}$ is type $SU(n)$ and $Z_{2}$ is either type $SU(n)$ or type $SU(n-1)$;

\item The symmetric space has restricted root system of type $D_{4}$, $L=2$, 
$Z_{1}$ is type $SU(4)$ and $Z_{2}$ is either type $SU(2)\times SU(2)$ and $%
\Phi _{Z_{2}}$ is Weyl conjugate to a subset of $\Phi _{Z_{1}}$, or $Z_{2}$
is type $SU(2)\times D_{2}$;

\item The symmetric space has restricted root system of type $D_{n}$, $n=3$
or $4,$ $L=3$ and $Z_{1},Z_{2},Z_{3}$ are all of type $SU(n)$ with Weyl
conjugate sets of annihilators in the case of $n=4$.
\end{enumerate}
\end{definition}

\subsection{Main Result}

Our main result is that other than for the exceptional tuples, eligibility
characterizes absolute continuity of the convolution product. The proof of
this theorem will occupy most of the remainder of the paper. Here is the
formal statement of the theorem.

\begin{theorem}
\label{main}Let $G/K$ by a symmetric space of type $III$ and suppose $%
Z_{j}\in \mathfrak{a}$, $Z_{j}\neq 0$ for $j=1,2,\dots ,L$ and $L\geq 2$.
The orbital measure $\mu _{Z_{1}}\ast \mu _{Z_{2}}\ast \cdots \ast \mu
_{Z_{L}}$ is absolutely continuous with respect to Lebesgue measure on $%
\mathfrak{p}$ if and only if $(Z_{1},Z_{2},\dots ,Z_{L})$ is eligible and
not exceptional.
\end{theorem}

\begin{corollary}
Let $Z_{j}\in \mathfrak{a}$, $Z_{j}\neq 0$ for $j=1,2,\dots ,L$ and $L\geq 2$%
, and let $z_{j}=\exp Z_{j}$. The orbital measure on $G$, $\nu _{z_{1}}\ast
\cdots \ast \nu _{z_{L}},$ is absolutely continuous with respect to Haar
measure on $G$ if and only if $(Z_{1},Z_{2},\dots ,Z_{L})$ is eligible and
not exceptional.
\end{corollary}

\begin{proof}
The proof is immediate from the Theorem and Cor. \ref{equiv}.
\end{proof}

\begin{remark}
The characterization of absolute continuity for pairs of orbital measures, $%
\nu _{x}\ast \nu _{y},$ was established by Gracyzk and Sawyer for the Type $%
III$ symmetric spaces of Cartan types $AI$ and $AII$ in \cite{GSLie} and for
the Cartan types $AIII$, $CII$ and $BDI$ in \cite{GSColloq} and \cite%
{GSArxiv}. They use an induction argument, but how it is applied depends
upon the particular symmetric space.

We will give a complete proof of sufficiency for all Cartan types and all $%
L\geq 2$. As with Gracyzk and Sawyer, we also use an induction argument, but
it relies upon the Lie type of the restricted root space rather than the
symmetric space itself. In fact, it is the combinatorial structure of the
root systems and root vectors that is key to our approach.
\end{remark}

\begin{corollary}
\label{A}(1) If $G/K$ is a symmetric space of Cartan type $AI$ or $AII$, of
rank $n$ (hence the restricted root system is type $A_{n}$), then the
convolution of any $n+1$ orbital measures (on $G$ or $\mathfrak{p}$) is
absolutely continuous. Moreover, this is sharp since any $n$-tuple of
elements all of type $SU(n)$ is not absolutely continuous. Furthermore,
these are the only $n$-tuples that fail to be absolutely continuous.

(2) If $G/K$ is a symmetric space of rank $n$ whose restricted root system
is not type $A_{n}$ or type $D_{3},$ then the convolution of any $n$ orbital
measures (on $G$ or $\mathfrak{p}$) is absolutely continuous. This is sharp
since any $(n-1)$-tuple of elements of type $B_{n-1}$ (or $C_{n-1}$, $%
D_{n-1} $, $BC_{n-1}$ depending on the restricted root system) is not
absolutely continuous. Furthermore, except in type $D_{4}$, these are the
only $(n-1)$-tuples that fail to be absolutely continuous.
\end{corollary}

We remind the reader that when we speak of an $L$-tuple of elements being
absolutely continuous, we mean that the convolution of their corresponding
orbital measures is absolutely continuous.

\begin{proof}
In both cases, just check the eligibility and non-exceptionality criterion.
\end{proof}

\begin{remark}
We remark that this corollary partially improves upon \cite{GSJFA} where it
was shown that in any symmetric space the convolution of $rank+1$ orbital
measures is absolutely continuous and that in the symmetric space with
restricted root system of type $A_{n}$, the $n$-fold convolution of the
orbital measure $\mu _{X},$ where $X$ is type $SU(n)$ is not absolutely
continuous.

This corollary also answers Conjecture 10 of \cite{GSJFA} negatively.
\end{remark}

A $K$-bi-invariant measure $\mu $ on $G$ is said to be \textit{continuous}
if $\mu (gK)=0$ for all $g\in G$. Ragozin in \cite{Ra} proved that the
convolution of any $\dim G/K$ continuous $K$-bi-invariant measures is
absolutely continuous. This too can be improved.

\begin{corollary}
If $G/K$ is a symmetric space of rank $n$, then the convolution of any $n$
(resp., $n+1$) continuous $K$-bi-invariant measures on $G$ is absolutely
continuous if the restricted root system is not type $A_{n}$ or type $D_{3}$
(resp., if the restricted root system is type $A_{n}$ or $D_{3}$).
\end{corollary}

\begin{proof}
In \cite{Ra} it was actually shown that if for each $Z_{1},...,Z_{m}\in 
\mathfrak{a}_{n}$, $sp\{Ad(k_{j})\mathcal{N}_{Z_{j}}:j=1,...,m\}=\mathfrak{p}
$ for almost all $k_{j}\in K$, then any $m$ continuous $K$-bi-invariant
measures on $G$ is absolutely continuous. From the Theorem and Prop. \ref%
{keyprop} we know this holds with $m=n$.
\end{proof}

Most of the remainder of the paper will be aimed at proving this theorem.
The proof is organized as follows. We focus first on sufficiency. We begin
by showing that the problem can largely be reduced to the study of the
problem on the symmetric spaces whose restricted root spaces are all of
dimension one. For these spaces we give an induction argument; this is is
the key combinatorial idea that was also used in the study of the analogous
problem for convolutions of orbital measures in the classical Lie algebras
(see \cite{GHAbs}). We will apply this first to the problem of convolving
two orbital measures and then will show how to handle more than two
convolutions.

Of course, an induction argument can only be used if we can establish the
base case(s). Some of these cases are non-trivial and for those we prove
another sufficient combinatorial condition that was motivated by a result in 
\cite{Wr}.

In passing from the symmetric spaces with one dimensional, restricted root
spaces to the other symmetric spaces, there are a few special cases of $L$%
-tuples that we will also need to handle using this other sufficient
condition.

We then turn to necessity. The necessity of eligibility will be seen to
follow from elementary linear algebra arguments. For the exceptional tuples,
we need other reasoning. The simple fact that the dimension of the
underlying orbits are simply not large enough to have a chance to satisfy
Prop. \ref{keyprop}(3) can often be used.

\section{Proof of Sufficiency}

\subsection{Reduction to multiplicities one problems}

We begin the proof of sufficiency by showing we can focus our attention
primarily on the symmetric spaces whose restricted root spaces are all of
dimension one. The first lemma seems to be known, but we could not find a
proof in the literature.

\begin{lemma}
\label{openinA}Let $G/K$ be a symmetric space and $x_{1},....,x_{m}\in A$.
Then $\nu _{x_{1}}\ast \cdot \cdot \cdot \ast \nu _{x_{m}}$ is absolutely
continuous on $G$ if and only if $\mathcal{A}(x_{1}Kx_{2}\cdot \cdot \cdot
Kx_{m})$ has non-empty interior in $A$.
\end{lemma}

\begin{proof}
Suppose $V$ is an open subset of $A$ contained in $\mathcal{A}%
(x_{1}Kx_{2}\cdot \cdot \cdot Kx_{m})$. Let $A^{+}=\exp \mathfrak{a}^{+}$.
Then $V\diagdown bdy(A^{+})\subseteq A^{+}$ is open in $A$. Since one can
easily check that $bdy(A^{+})$ has $A$-Haar measure zero,\textbf{\ } $%
V\diagdown bdy(A^{+})$ is non-empty.

It is known that the map $\mathcal{A}$ restricted to $KA^{+}K$ is a smooth
map onto $A^{+}$ \textbf{(\cite{GS2003})}, thus the preimage of $V\diagdown
bdy(A^{+})$ is open in $KA^{+}K$ and hence also in $G$ since $KA^{+}K$ is
open in $G$. But this non-empty open set is a subset of $Kx_{1}K\cdot \cdot
\cdot Kx_{m}K$ and therefore by Prop. \ref{keyprop}, $\nu _{x_{1}}\ast \cdot
\cdot \cdot \ast \nu _{x_{m}}$ is absolutely continuous.

Conversely, suppose $\nu _{x_{1}}\ast \cdot \cdot \cdot \ast \nu _{x_{m}}$
is absolutely continuous on $G$. Applying Prop. \ref{keyprop} we can find an
open, non-empty subset $V$ of $G$ contained in $Kx_{1}K\cdot \cdot \cdot
Kx_{m}K.$ But then also $KVK$ is an open set in $G$ contained in $%
Kx_{1}K\cdot \cdot \cdot Kx_{m}K$ and hence $KVK\bigcap A$ is open in $A$
(the topology on $A$ being the relative topology). It is non-empty since
every double coset admits elements of $A$.

Now%
\begin{equation*}
KVK\bigcap A=\bigcup_{w\in W}\left( KVK\bigcap w(A^{+})\right) \bigcup
\bigcup_{w\in W}\left( KVK\bigcap bdy(w(A^{+}))\right) .
\end{equation*}%
The sets $KVK\bigcap w(A^{+})$ are all open in $A$. If for some $w\in W$, $%
KVK\bigcap w(A^{+})$ is non-empty, then since $w^{-1}(KVK)=KVK$, it would
follow that $KVK\bigcap A^{+}$ $\subseteq Kx_{1}K\cdot \cdot \cdot Kx_{m}K$
is open and non-empty$.$ But then $\mathcal{A}(KVK\bigcap A^{+})$ is also
open and non-empty, hence $\mathcal{A}(Kx_{1}K\cdot \cdot \cdot Kx_{m}K)=%
\mathcal{A}(x_{1}K\cdot \cdot \cdot Kx_{m})$ has non-empty interior, as we
desired to show.

Otherwise, $KVK\bigcap A=\bigcup_{w\in W}\left( KVK\bigcap
bdy(w(A^{+}))\right) $. But the set on the right has Haar measure zero,
while the set on the left is open and non-empty, so this is impossible.
\end{proof}

\textbf{Terminology:} Let $G_{1}/K_{1}$ and $G_{2}/K_{2}$ be two symmetric
spaces. We say that $G_{1}/K_{1}$ is embedded into $G_{2}/K_{2}$ if there is
a mapping $\mathcal{I}:G_{1}\rightarrow G_{2}$ satisfying the following
properties.

\begin{definition}
\begin{enumerate}
\item $\mathcal{I}$ is a group isomorphism into $G_{2}$.

\item $\mathcal{I}$ restricted to $A_{1}$ is a topological group isomorphism
onto $A_{2}$.

\item $\mathcal{I}$ maps $K_{1}$ into $K_{2}$.
\end{enumerate}
\end{definition}

Property (2) ensures that the symmetric spaces have the same rank. Here are
some examples of embeddings.

\begin{lemma}
\label{embedding1}In the following cases $G_{1}/K_{1}$ embeds into $%
G_{2}/K_{2}$:%
\begin{equation*}
\begin{array}{cccccc}
\begin{array}{c}
\text{Cartan } \\ 
\text{class}%
\end{array}
& G_{1} & K_{1} & G_{2} & K_{2} & 
\begin{array}{c}
\text{Cartan } \\ 
\text{class}%
\end{array}
\\ 
AI & SL(n,\mathbb{R)} & SO(n) & SL(n,\mathbb{H)} & Sp(n) & AII \\ 
BDI & SO_{0}(p,q),q\geq p & SO(p)\times SO(p) & SO_{0}(p,q+1),q\geq p & 
SO(p)\times SO(q) & BDI \\ 
BDI & SO_{0}(p,q),q\geq p & SO(p)\times SO(q) & SU(p,q),q\geq p & 
SU(p)\times SU(q) & AIII \\ 
AIII & SU(p,q),q\geq p & SU(p)\times SU(q) & Sp(p,q),q\geq p & Sp(p)\times
Sp(q) & CII \\ 
\text{Type }IV & SO(n,\mathbb{C)} & SO(n) & SO^{\ast }(2n) & U(n) & DIII%
\end{array}%
\end{equation*}
\end{lemma}

\begin{proof}
In fact, it is obvious that in all but the last case that the embedding map $%
\mathcal{I}$ is the identity and $A_{1}=A_{2}$.

For the final case, we remind the reader that $SO(n,\mathbb{C)}$ is the set
of $n\times n$ complex matrices $g$ satisfying $g^{t}g=Id$ and $SO^{\ast
}(2n)$ are the matrices in $SO(2n,\mathbb{C)}$ with the additional
requirement that $g^{t}J_{n}\overline{g}=J_{n}$ where $J_{n}=\left[ 
\begin{array}{cc}
0 & I_{n} \\ 
-I_{n} & 0%
\end{array}%
\right] $. The subgroup $SO(n)$ embeds into $SO(n,\mathbb{C)}$ in the
natural way and $U(n)$ embeds into $SO^{\ast }(2n)$ as follows: The matrix $%
X+iY\in U(n)$, where $X,Y$ are real, maps to $\left[ 
\begin{array}{cc}
X & Y \\ 
-Y & X%
\end{array}%
\right] $. Here $A_{1}=\{\exp iX:X\in \mathfrak{t}\}$ where $\mathfrak{t}$
is the maximal torus of the Lie algebra of $SO(n)$.

If we define $\mathcal{I}$ $:SO(n,\mathbb{C)\rightarrow }SO^{\ast }(2n)$ by $%
\mathcal{I}(g)=\left[ 
\begin{array}{cc}
g & 0 \\ 
0 & \overline{g}%
\end{array}%
\right] $, then $\mathcal{I}(A_{1})=A_{2}$ and the other conditions of the
embedding lemma are also satisfied.
\end{proof}

The embedding property is important because we can deduce absolute
continuity of certain convolution products of orbital measures in the
`larger' space $G_{2}/K_{2}$ from the property in the `smaller' symmetric
space $G_{1}/K_{1}$.

\begin{proposition}
\label{embedding2}Suppose $G_{1}/K_{1}$ is embedded into $G_{2}/K_{2}$ with
the mapping $\mathcal{I}$.

\begin{enumerate}
\item Let $x_{1},....,x_{m}\in A_{1}$. If $\nu _{x_{1}}\ast \cdot \cdot
\cdot \ast \nu _{x_{m}}$ is absolutely continuous on $G_{1}$ and $z_{j}=%
\mathcal{I}(x_{j})$, then $\nu _{z_{1}}\ast \cdot \cdot \cdot \ast \nu
_{z_{m}}$ is absolutely continuous on $G_{2}$.

\item Let $X_{1},....,X_{m}\in \mathfrak{a}_{1}$. If $(X_{1},...,X_{m})$ is
an absolutely continuous tuple on $\mathfrak{p}_{1}$ and $\exp (Z_{j})=%
\mathcal{I}(\exp X_{j})$, then $(Z_{1},...,Z_{m})$ is absolutely continuous
on $\mathfrak{p}_{2}$.
\end{enumerate}
\end{proposition}

\begin{proof}
Put $\mathcal{A}_{j}:G_{j}\rightarrow A_{j}$ for $j=1,2$. One can check from
the definitions that $\mathcal{I\circ A}_{1}=\mathcal{A}_{2}\circ \mathcal{I}
$. As $\mathcal{I}$ is a group isomorphism, for all $k_{1},...,k_{m-1}$ $\in
K_{1}$ we have 
\begin{eqnarray*}
\mathcal{I}(\mathcal{A}_{1}(x_{1}k_{1}\cdot \cdot \cdot k_{m-1}x_{m})) &=&%
\mathcal{A}_{2}(\mathcal{I}(x_{1})\mathcal{I}(k_{1})\mathcal{I}(x_{2})\cdot
\cdot \cdot \mathcal{I}(k_{m-1})\mathcal{I}(x_{m})) \\
&\subseteq &\mathcal{A}_{2}(z_{1}K_{2}z_{2}\cdot \cdot \cdot K_{2}z_{m}).
\end{eqnarray*}%
Hence $\mathcal{I}(\mathcal{A}_{1}(x_{1}K_{1}\cdot \cdot \cdot
K_{1}x_{m}))\subseteq \mathcal{A}_{2}(z_{1}K_{2}z_{2}\cdot \cdot \cdot
K_{2}z_{m})$.

As $\nu _{x_{1}}\ast \cdot \cdot \cdot \ast \nu _{x_{m}}$ is absolutely
continuous on $G_{1}$, Lemma \ref{openinA} implies there is an open set $%
V\subseteq A_{1}$ with $V\subseteq \mathcal{A}_{1}(x_{1}K_{1}x_{2}\cdot
\cdot \cdot K_{1}x_{m})$. Since $\mathcal{I}(V)$ is open in $A_{2}$ and
contained in $\mathcal{A}_{2}(z_{1}K_{2}z_{2}\cdot \cdot \cdot K_{2}z_{m})$,
it follows by another application of the lemma that $\nu _{z_{1}}\ast \cdot
\cdot \cdot \ast \nu _{z_{m}}$ is absolutely continuous on $G_{2}$.

Part (2) follows from (1) and Corollary \ref{equiv}.
\end{proof}

\begin{remark}
This idea is implicit in the work of Graczyk and Sawyer, in the special case
of the embedding map being the identity.
\end{remark}

\subsection{Induction argument}

Let $G_{n}/K_{n}$ be a symmetric space of rank $n$, with restricted root
system of type $B_{n}\,$, $C_{n},D_{n}$ or $BC_{n}$. Let $Z\in \mathfrak{a}%
_{n}$, say%
\begin{equation*}
Z=(\underbrace{0,\dots ,0}_{J},\underbrace{a_{1},\dots ,a_{1}}_{s_{1}},\dots
,\underbrace{a_{m},\dots ,(\pm )a_{m}}_{s_{m}})\in \mathfrak{a}_{n},
\end{equation*}%
where $s_{1}=\max s_{j}$. We denote by $Z^{\prime }$ the element of $%
\mathfrak{a}_{n-1}$ given by 
\begin{equation}
Z^{\prime }=%
\begin{cases}
(\underbrace{0,\dots ,0}_{J-1},\underbrace{a_{1},\dots ,a_{1}}_{s_{1}},\dots
,\underbrace{a_{m},\dots ,(\pm )a_{m}}_{s_{m}}) & \text{ if }2J\geq s_{1} \\ 
(\underbrace{0,\dots ,0}_{J},\underbrace{a_{1},\dots ,a_{1}}_{s_{1}-1},\dots
,\underbrace{a_{m},\dots ,(\pm )a_{m}}_{s_{m}}) & \text{if }2J<s_{1}%
\end{cases}%
.  \label{X'}
\end{equation}%
Define $Z^{\prime }$ similarly when the restricted root system of $%
G_{n}/K_{n}$ is type $A_{n}$.

We embed $\mathfrak{a}_{n-1}$ into $\mathfrak{a}_{n}$ by taking the standard
basis vectors $e_{1},\dots ,e_{n}$ in $\mathbb{R}^{n}$ (or $%
e_{1}-e_{n+1},\dots ,e_{n}-e_{n+1}$ in $\mathbb{R}^{n+1}$ in the case of
type $A_{n})$ as the basis for $\mathfrak{a}_{n}$ and taking the vectors $%
e_{2},\dots ,e_{n}$ (resp., $e_{2}-e_{n+1},\dots ,e_{n}-e_{n+1})$ as the
basis for $\mathfrak{a}_{n-1}$. This also gives a natural embedding of $\Phi
_{n-1}$ into $\Phi _{n}$ and together these give an embedding of $\mathfrak{g%
}_{n-1}$, $\mathfrak{p}_{n-1}$ and $\mathfrak{k}_{n-1}$ into $\mathfrak{g}%
_{n}$, $\mathfrak{p}_{n}$ and $\mathfrak{k}_{n}$ respectively, an embedding
of $G_{n-1}$ into $G_{n}$, and an embedding of $K_{n-1}$ into $K_{n}$. We
will also view $Z^{\prime }$ as an element of $\mathfrak{a}_{n}$ in the
natural way.

With this understanding, put 
\begin{equation*}
\Omega _{Z}=\mathcal{N}_{Z}\diagdown \mathcal{N}_{Z^{\prime }}\subseteq 
\mathfrak{p}_{n}.
\end{equation*}

\begin{lemma}
\label{eligible}If $(X,Y)$ is an eligible pair in $\mathfrak{a}_{n}$ and $%
X,Y $ are not both of type $SU(m)\times SU(m)$ in a symmetric space with
restricted root system of type $A_{n}$ where $n=2m-1$, then the reduced
pair, $(X^{\prime },Y^{\prime }),$ is eligible in $\mathfrak{a}_{n-1}$.
\end{lemma}

\begin{proof}
The proof is a straightforward calculation.\textbf{\ }The details are worked
out for the compact Lie algebra case in \cite[Lemma 3]{GHAbs}.
\end{proof}

We next adapt the general strategy used in \cite[Prop. 2]{GHAbs} for the
corresponding problem in the classical compact Lie algebra setting.

\begin{proposition}
\label{general strategy}\textrm{(}General Strategy\/\textrm{) } Let $%
G_{n}/K_{n}$ be a symmetric space of rank $n,$ with associated Lie algebra $%
\mathfrak{g}_{n}=\mathfrak{k}_{n}\oplus \mathfrak{p}_{n}$ and maximal
abelian subspace $\mathfrak{a}_{n}$. Let $X_{i}\in \mathfrak{a}_{n}$, $%
i=1,\dots ,L$ for $L\geq 2$ and assume $(X_{1}^{\prime },\dots
,X_{L}^{\prime })$ is an absolutely continuous tuple in $\mathfrak{p}_{n-1}$.

Let $\mathcal{V}_{j}=\mathfrak{p}_{j}\ominus \mathfrak{a}_{j}$ for $j=n-1,n$%
. Suppose $\Omega $ is a subset of $\mathcal{V}_{n}\ominus \mathcal{V}_{n-1}$
that contains all $\Omega _{X_{i}}$ and has the property that $ad(H)(\Omega
)\subseteq sp\Omega $ whenever $H\in \mathfrak{k}_{n-1}$. Fix $\Omega
_{0}\subseteq \Omega _{X_{L}}$ and assume there exists $k_{1},\dots
,k_{L-1}\in K_{n-1}$ and $M\in \mathfrak{k}_{n}$ such that

\textrm{(i)} $sp\{Ad(k_{i})(\Omega _{X_{i}}),\Omega _{X_{L}}\backslash
\Omega _{0}:i=1,\dots ,L-1\}=sp\Omega ;$

\textrm{(ii)} $ad^{k}(M):\mathcal{N}_{X_{L}}\backslash \Omega
_{0}\rightarrow sp\{\Omega ,\mathfrak{p}_{n-1}\}$ for all positive integers $%
k$; and

\textrm{(iii)} The span of the projection of $Ad(\exp sM)(\Omega _{0})$ onto
the orthogonal complement of $sp\{\mathfrak{p}_{n-1},\Omega \}$ in $%
\mathfrak{p}_{n}$ is a surjection for all small $s>0.$

Then $(X_{1},\dots ,X_{L})$ is an absolutely continuous tuple in $\mathfrak{p%
}_{n}$.
\end{proposition}

\begin{proof}
As $(X_{1}^{\prime },\dots ,X_{L}^{\prime })$ is an absolutely continuous
tuple, Prop. \ref{keyprop} tells us that 
\begin{equation*}
sp\left\{ Ad(h_{i})(\mathcal{N}_{X_{i}^{\prime }}),\mathcal{N}%
_{X_{L}^{\prime }}:i=1,\dots ,L-1\right\} =\mathfrak{p}_{n-1}
\end{equation*}%
for a dense set of $(h_{1},\dots ,h_{L-1})\in K_{n-1}^{L-1}$. Given $%
\varepsilon >0$, choose such $h_{i}=h_{i}(\varepsilon )\in K_{n-1}$ with $%
\left\Vert Ad(h_{i})-Ad(k_{i})\right\Vert <\varepsilon ,$ where the elements 
$k_{i}\in K_{n-1}$ are the ones given in the hypothesis of the proposition
and the norm is the operator norm.

An elementary linear algebra argument, together with assumption (i), shows
that for sufficiently small $\varepsilon >0$, 
\begin{equation*}
\dim (sp\Omega )=\dim \left( sp\{Ad(h_{i})(\Omega _{X_{i}}),\Omega
_{X_{L}}\backslash \Omega _{0}:i=1,\dots ,L-1\}\right) .
\end{equation*}

Since $ad(H)(\Omega )\subseteq sp(\Omega )$ for all $H\in \mathfrak{k}_{n-1}$
and $h_{i}=\exp H_{i}$ for some $H_{i}\in \mathfrak{k}_{n-1}$ we have $%
Ad(h_{i})(\Omega )=\exp (ad(H_{i})(\Omega )\subseteq sp\Omega $ for all $%
h_{i}\in K_{n-1}$. Thus for sufficiently small $\varepsilon >0$, 
\begin{equation*}
sp\left\{ Ad(h_{i})(\Omega _{X_{i}}),\Omega _{X_{L}}\backslash \Omega
_{0}:i=1,\dots ,L-1\right\} =sp\Omega
\end{equation*}%
For such a choice of $\varepsilon $ (now fixed) we have 
\begin{equation*}
sp\{Ad(h_{i})(\mathcal{N}_{X_{i}}),\mathcal{N}_{X_{L}}\backslash \Omega
_{0}:i=1,\dots ,L-1\}=sp\{\Omega ,\mathfrak{p}_{n-1}\}.
\end{equation*}

Assumption (ii) and the fact that $\mathcal{N}_{X_{L}}\backslash \Omega
_{0}\subseteq sp\{\Omega ,\mathfrak{p}_{n-1}\}$ implies that for any $s>0$, $%
\exp (s\cdot adM)=Ad(\exp sM)$ maps $\mathcal{N}_{X_{L}}\backslash \Omega
_{0}$ to $sp\{\Omega ,\mathfrak{p}_{n-1}\}$. Moreover, $\left\Vert
Id-Ad(\exp sM)\right\Vert \rightarrow 0$ as $s\rightarrow 0$, thus similar
reasoning to that above shows that for all small enough $s>0$, 
\begin{equation*}
sp\{\Omega ,\mathfrak{p}_{n-1}\}=sp\{Ad(h_{i})(\mathcal{N}_{X_{i}}),Ad(\exp
sM)(\mathcal{N}_{X_{L}}\backslash \Omega _{0}):i=1,\dots ,L-1\}=\mathfrak{p}%
_{n},
\end{equation*}%
with the final equality coming from (iii). Another application of Prop.~\ref%
{keyprop} proves that $\mu _{X_{1}}\ast \cdots \ast \mu _{X_{L}}$ is
absolutely continuous.
\end{proof}

We will now focus on the symmetric spaces all of whose restricted root
spaces have dimension one. These are the symmetric spaces of Cartan type $AI$%
, $CI$, $DI$ and $BI$, the latter two in the cases when the symmetric space
is $SO_{0}(p+q)/SO(p)\times SO(q)$ with $q=p$ and $q=p+1$ respectively.

For such spaces we will introduce the notation $E_{\alpha }$ for a (fixed
choice of) basis vector of the restricted root space $\mathfrak{g}_{\alpha }$%
, $\alpha \in \Phi ^{+}$, and put 
\begin{eqnarray*}
E_{\alpha }^{+} &=&E_{\alpha }+\theta E_{\alpha } \\
E_{\alpha }^{-} &=&E_{\alpha }-\theta E_{\alpha }.
\end{eqnarray*}

The following is well known and very important for us:%
\begin{equation}
\lbrack E_{\alpha }^{+},E_{\beta }^{-}]=cE_{\alpha +\beta }^{-}+dE_{\alpha
-\beta }^{-}  \label{adrule}
\end{equation}%
where $c$ (or $d$) $\neq 0$ if $\alpha +\beta $ (respectively, $\alpha
-\beta )$ is a restricted root and $E_{\alpha +\beta }^{-}$ (or $E_{\alpha
-\beta }^{-}$) should be understood to be the zero vector if $\alpha +\beta $
(resp. $\alpha -\beta )$ is not a restricted root. Furthermore, 
\begin{equation*}
0\neq \lbrack E_{\alpha }^{+},E_{\alpha }^{-}]\in \mathfrak{a}
\end{equation*}%
and if for some subset of roots $I$, $\{\alpha :\alpha \in I\}$ spans $%
sp\Phi $, then $\{$ $[E_{\alpha }^{+},E_{\alpha }^{-}]:\alpha \in I\}$ spans 
$\mathfrak{a}$.

Here is the key induction argument, the most significant step in the proof
of sufficiency.

\begin{theorem}
\label{indstep}Assume $G_{n}/K_{n}$ is a Type $III$ symmetric space of rank $%
n$, whose restricted root spaces all have dimension one. Suppose $(X,Y)$ is
an eligible pair in $\mathfrak{a}_{n}$. Assume that $X$ and $Y$ are not both
of type $SU(n)$ when the restricted root system is type $D_{n}$ and are not
both of type $SU(m)\times SU(m)$ for $n=2m-1$ when the restricted root
system is type $A_{2m-1}$. Assume, also, that the reduced pair, $(X^{\prime
},Y^{\prime }),$ is an absolutely continuous pair in $\mathfrak{p}_{n-1}$.
Then $(X,Y)$ is an absolutely continuous pair in $\mathfrak{p}_{n}.$
\end{theorem}

\begin{proof}
As mentioned above, these are the symmetric spaces of Cartan type $AI$, $CI$%
, $DI$ ($q=p$) and $BI$ ($q=p+1$) and hence their restricted root systems
are types $A_{n}$, $C_{n}$, $D_{n}$ (with $n=p$) and $B_{n}$ (with $n=p$),
respectively.

The proof of the theorem is essentially the same as that given in \cite[%
Prop. 3]{GHAbs} for orbital measures in the classical, compact Lie algebras,
with an appropriate change of notation. But as the ideas are so important
for this paper we will present a condensed overview of the arguments for the
restricted root spaces of type $B_{n},C_{n}$ or $D_{n}$. Type $A_{n}$ is
similar to case 1(a) below, but easier, and is left for the reader.

Different arguments will be needed depending on the dominant type of $X$ and 
$Y$.

Case 1: Neither $X$ nor $Y$ are dominant $SU$ type.

Suppose $S_{X}=2J$ and $S_{Y}=2N$. Applying a Weyl conjugate if necessary
(which corresponds to the Ad-action of an element in $K$) we can assume,
without loss of generality, that%
\begin{equation*}
\Omega _{X}=\{E_{e_{1}\pm e_{j}}^{-}:j>J\}\text{ and }\Omega
_{Y}=\{E_{e_{1}\pm e_{j}}^{-}:j>N\}.
\end{equation*}

Case 1(a): The restricted root system\ is type $D_{n}$. Put%
\begin{equation*}
\Omega =\{E_{e_{1}\pm e_{j}}^{-}:j=2,...,n\}\text{ and }\Omega
_{0}=\{E_{e_{1}+e_{n}}^{-}\}.
\end{equation*}%
Property (\ref{adrule}) ensures that $ad(H)(\Omega )\subseteq sp\Omega $
whenever $H\in \mathfrak{k}_{n-1}$.

Take $k\in K_{n-1}$ the Weyl conjugate that permutes the letters $1+j$ with $%
N+j$ for $j=1,...,J-1$. The eligibility assumption ensures $\{Ad(k)(\Omega
_{Y}),$ $\Omega _{X}\diagdown \Omega _{0}\}=\Omega $, thus Prop. \ref%
{general strategy}(i) is satisfied.

Set $M=E_{e_{1}+e_{n}}^{+}\in \mathfrak{k}_{n}$ and note that Prop. \ref%
{general strategy}(ii) is also met.

The projection of $ad(M)(E_{e_{1}+e_{n}}^{-})$ maps onto the orthogonal
complement of $sp\{\Omega ,\mathfrak{p}_{n-1}\}$ in $\mathfrak{p}_{n}$ since 
$sp\{\Omega ,\mathfrak{p}_{n-1}\}$ is of co-dimension one, thus (iii) is
also fulfilled with any $s>0$. Applying Prop. \ref{general strategy}, we
conclude that $\mu _{X}\ast \mu _{Y}$ is absolutely continuous.

Case 1(b): The restricted root space is type $B_{n}$.The arguments are
similar. Take%
\begin{equation*}
\Omega =\{E_{e_{1}\pm e_{j}}^{-},E_{e_{1}}^{-}:j=2,...,n\}\text{ and }\Omega
_{0}=\{E_{e_{1}+e_{n}}^{-}\}.
\end{equation*}%
Let $k\in K_{n-1}$ be the Weyl conjugate that permutes the letters $1+j$
with $N+j$ for $j=1,...,J-1$, as in the previous case, and let $k_{t}=\left(
\exp tE_{e_{n}}^{+}\right) k\in K_{n-1}$ for small $t>0$. Since 
\begin{equation*}
Ad(k_{t})(E_{e_{1}\pm e_{n}}^{-})=a(t)E_{e_{1}\pm
e_{n}}^{-}\;+tb(t)E_{e_{1}}^{-}\;+t^{2}c(t)E_{e_{1}\mp e_{n}}^{-}
\end{equation*}%
where $a(t)\rightarrow 1$ as $t\rightarrow 0$, and $b(t)$, $c(t)$ converge
to non-zero scalars, one can deduce that 
\begin{equation*}
sp\{Ad(k)(\Omega _{Y}),\text{ }\Omega _{X}\diagdown \Omega _{0}\}=sp\Omega 
\text{.}
\end{equation*}%
Now take $M=E_{e_{1}+e_{n}}^{+}$ and apply the general strategy.

Case 1(c): The restricted root space is type $C_{n}$. Here we begin with 
\begin{equation*}
\Omega =\{E_{e_{1}\pm e_{j}}^{-}:j=2,...,n\}\text{ and }\Omega
_{0}=\{E_{e_{1}+e_{n}}^{-}\},
\end{equation*}%
and let $k\in K_{n-1}$ be the Weyl conjugate that permutes the letters $1+j$
with $N+j$ for $j=1,...,J-1$. As in the proof of the general strategy, the
induction assumption implies there is some $h\in K_{n-1}$ such that 
\begin{equation*}
sp\{Ad(h)\mathcal{N}_{Y},\mathcal{N}_{X}\diagdown \Omega _{0}\}=sp\{\Omega ,%
\mathfrak{p}_{n-1}\}\text{.}
\end{equation*}

We again take $M=E_{e_{1}+e_{n}}^{+}\in \mathfrak{k}_{n}$, but in this case
cannot call directly upon the general strategy as it is not true that $%
ad^{k}(M)(H)\in sp\{\Omega ,\mathfrak{p}_{n-1}\}$ for all $H\in \mathcal{N}%
_{X}\diagdown \Omega _{0}$. However, one can check that for small $s>0$, 
\begin{equation*}
sp\{E_{e_{1}-e_{n}}^{-},Ad(\exp
sM)\{E_{e_{1}-e_{n}}^{-},E_{2e_{n}}^{-}\}\}=sp%
\{E_{e_{1}-e_{n}}^{-},E_{2e_{1}}^{-},E_{2e_{n}}^{-}\}
\end{equation*}%
and using this fact it can be shown that%
\begin{equation*}
sp\{Ad(h)\mathcal{N}_{Y},Ad(\exp sM)(\mathcal{N}_{X}\diagdown \Omega _{0})\}=%
\mathfrak{p}_{n}\ominus sp(a_{1})
\end{equation*}%
where $a_{1}$ is the standard basis vector of $\mathfrak{a}_{n}\ominus 
\mathfrak{a}_{n-1}$. Standard arguments then allow one to deduce that for
small enough $s>0$, 
\begin{equation*}
sp\{Ad(h)\mathcal{N}_{Y},Ad(\exp sM)(\mathcal{N}_{X})\}=\mathfrak{p}_{n}%
\text{.}
\end{equation*}%
For the details of this technical argument we refer the reader to the proof
of Prop. 3 Case 1(c) in \cite{GHAbs}.

Case 2: Both $X,Y$ are dominant type $SU$. First, assume the restricted root
space is either $B_{n}$ or $C_{n}$. Let 
\begin{equation*}
\Omega =\{E_{e_{1}\pm e_{j}}^{-},E_{(2)e_{1}}^{-}:j=2,...,n\}\text{ and }%
\Omega _{0}=\{E_{(2)e_{1}}^{-}\}
\end{equation*}%
(with the choice $E_{2e_{1}}^{-}$ or $E_{e_{1}}^{-}$ depending on whether
the underlying root system is type $C_{n}$ or $B_{n}$). $\ $Applying a Weyl
conjugate change of sign, as needed, there is no loss of generality in
assuming $\Omega _{X}$ contains all $E_{e_{1}+e_{j}}^{-}$ for $j\geq 2$ and $%
E_{(2)e_{1}}^{-},$ and $\Omega _{Y}$ contains all $E_{e_{1}-e_{j}}^{-}$ for $%
j\geq 2$ and (again) $E_{(2)e_{1}}^{-}$. Hence $\{\Omega _{Y},$ $\Omega
_{X}\diagdown \Omega _{0}\}=\Omega $. Now take $M=E_{(2)e_{1}}^{+}$.

If, instead, the resricted root space is $D_{n}$, then we take $\Omega
=\{E_{e_{1}\pm e_{j}}^{-},:j=2,...,n\}$. Since $X$ and $Y$ are not both of
type $SU(n)$, we can assume $\Omega _{X}$ contains all the roots $%
E_{e_{1}+e_{j}}^{-}$ for $j\geq 2$ and both $E_{e_{1}\pm e_{n}}^{-},$ and $%
\Omega _{Y}$ contains all $E_{e_{1}-e_{j}}^{-}$ for $2\leq j\leq n-1$ and at
least one of $E_{e_{1}\pm e_{n}}^{-}$. Take $\Omega _{0}$ to be the one of $%
E_{e_{1}\pm e_{n}}^{-}$ that belongs to $\Omega _{Y}$ and argue as above.

Case 3: $X,Y$ are of different dominant type, say $X$ is dominant $SU$ type
with $S_{X}=m$ and $S_{Y}=2J$. Take $\Omega =\mathcal{V}_{n}\ominus \mathcal{%
V}_{n-1}$. Applying suitable Weyl conjugates, we can assume 
\begin{eqnarray*}
\Omega _{X}
&=&\{E_{e_{1}+e_{j}}^{-},E_{e_{1}-e_{k}}^{-},E_{(2)e_{1}}^{-}:j\geq 2,k>m\}%
\text{ and} \\
\Omega _{Y} &=&\{E_{e_{1}\pm e_{j}}^{-}:2\leq j\leq n-J+1\}
\end{eqnarray*}%
(with appropriate modifications in type $D_{n}$). Put 
\begin{equation*}
\Omega _{0}=\{E_{e_{1}+e_{n-J+1}}^{-}\}\subseteq \Omega _{X}\cap \Omega _{Y}.
\end{equation*}%
If $n-J+1\geq m$, then $\{\Omega _{Y},\Omega _{X}\diagdown \Omega
_{0}\}=\Omega $ and the rest of the argument is easy. Otherwise, let 
\begin{equation*}
\Omega _{1}=\{E_{e_{1}+e_{k}}^{-}:2\leq k\leq n-J\}\subseteq (\Omega
_{Y}\cap \Omega _{X})\diagdown \Omega _{0}.
\end{equation*}%
Define%
\begin{equation*}
H=\sum_{j=2}^{m+J-n}E_{e_{j}+e_{n-J+j}}^{+}.
\end{equation*}%
The eligibility condition implies $\Omega _{1}\supseteq
\{E_{e_{1}+e_{k}}^{-}:k\leq m+J-n\}$. It follows that%
\begin{equation*}
sp\{ad(H)(\Omega _{1}),\Omega _{Y}\diagdown \Omega _{1},\Omega _{X}\diagdown
\Omega _{0}\}=sp\Omega .
\end{equation*}%
A linear algebra argument implies there is some $k\in K_{n-1}$ (namely, $%
k=\exp tH$ for sufficiently small $t>0$) such that%
\begin{equation*}
sp\{Ad(k)(\Omega _{Y}),\Omega _{X}\diagdown \Omega _{0}\}=sp\Omega .
\end{equation*}%
Now, take $M=E_{e_{1}+e_{n-J+1}}^{+}$ and apply the general strategy to
complete the argument.
\end{proof}

\subsection{Another sufficient condition}

To use the induction argument outlined in the previous subsection we will,
of course, need to do the base cases. A sufficient condition for absolute
continuity that will be helpful to us for in doing this (and also for
dealing with some special tuples when the restricted root spaces are higher
dimensional) is the following variant of a result of \cite{Wr}.

By the rank of a root subsystem we mean the dimension of the Euclidean space
it spans. By the dimension of a root subsystem $\Phi _{0}$, we mean 
\begin{equation*}
\dim \Phi _{0}:=\dim sp\{X_{\alpha }^{-}:X_{\alpha }\in \mathfrak{g}_{\alpha
}\text{, }\alpha \in \Phi _{0}\}.
\end{equation*}%
This is the cardinality of $\Phi _{0}^{+}$ counted by multiplicity of the
corresponding restricted root spaces. When the multiplicities of all the
restricted root spaces coincide, say are equal to $r$, then $\dim \Phi
_{0}=r\cdot card(\Phi _{0}^{+})$.

\begin{theorem}
\label{WrCriteria} Assume $G/K$ is a symmetric space with restricted root
system $\Phi $ and Weyl group $W$. Suppose $Z_{1},...,Z_{m}$ $\in \mathfrak{a%
}$. Assume%
\begin{equation}
(m-1)\left( \dim \Phi -\dim \Psi \right) -1\geq \sum_{i=1}^{m}\left( \dim
\Phi _{Z_{i}}-\min_{\sigma \in W}\dim (\Phi _{Z_{i}}\cap \sigma (\Psi
))\right)  \label{WrC}
\end{equation}%
for all root subsystems $\Psi \subseteq \Phi $ of co-rank $1$ and having the
property that $sp(\Psi )\cap \Phi =\Psi $. Then $\mu _{Z_{1}}\ast \cdot
\cdot \cdot \ast \mu _{Z_{m}}$ is absolutely continuous.
\end{theorem}

We first prove several lemmas. Throughout, $Z_{1},...,Z_{m}\in \mathfrak{a}$
will be fixed. For the proof, denote by%
\begin{equation*}
n_{X}=\ker (adX)|_{\mathfrak{p}}=\{Y\in \mathfrak{p}:[X,Y]=0\}.
\end{equation*}

\begin{lemma}
The sum $O_{Z_{1}}+\cdot \cdot \cdot +O_{Z_{m}}$ has non-empty interior in $%
\mathfrak{p}$ if and only if there exist $k_{1},...,k_{m}\in K$ such that 
\begin{equation*}
\bigcap\limits_{j=1}^{m}Ad(k_{j})n_{Z_{j}}=\{0\}.
\end{equation*}
\end{lemma}

\begin{proof}
This is a Hilbert space argument taking the inner product given by the
Killing form. From Prop. \ref{keyprop}, $O_{Z_{1}}+\cdot \cdot \cdot
+O_{Z_{m}}$ has non-empty interior if and only if there is some $k_{1}=Id$, $%
k_{2},...,k_{m}\in K$ such that 
\begin{equation}
sp\{Ad(k_{j})\mathcal{N}_{Z_{j}}:j=1,...,m\}=\mathfrak{p.}  \label{ref1}
\end{equation}%
We note that $sp\mathcal{N}_{Z_{j}}=\text{Im } ad(Z_{j})|_{\mathfrak{k}}$.
Thus (\ref{ref1}) holds if and only if 
\begin{eqnarray*}
\mathfrak{p} &\mathfrak{=}&\sum_{j=1}^{m}Ad(k_{j})\text{Im }ad(Z_{j})|_{%
\mathfrak{k}}=\sum_{j=1}^{m}Ad(k_{j})\left( \ker (ad(Z_{j})|_{\mathfrak{p}%
})\right) ^{\bot } \\
&=&\sum_{j=1}^{m}Ad(k_{j})\left( n_{Z_{j}}\right) ^{\bot }=\left(
\bigcap\limits_{j=1}^{m}Ad(k_{j})n_{Z_{j}}\right) ^{\bot },
\end{eqnarray*}%
where the orthogonal complements are all understood to be in $\mathfrak{p}$.
\end{proof}

We call $Z\in \mathfrak{p}$ \textit{maximally singular} if whenever $%
W=Ad(k)Z\in \mathfrak{a}$ for $k\in K$, then $\Phi _{W}$ is of co-rank one.
This is equivalent to saying $O_{Z}$ contains an element in $\mathfrak{a}$
whose set of annihilating roots is a co-rank one root subsystem.

\begin{lemma}
If the intersection 
\begin{equation*}
\bigcap\limits_{j=1}^{m}Ad(k_{j})n_{Z_{j}}\neq \{0\}
\end{equation*}%
for some $k_{j}\in K$, then the intersection contains a maximally singular
element.
\end{lemma}

\begin{proof}
Suppose $Z\in \bigcap\limits_{j=1}^{m}Ad(k_{j})n_{Z_{j}}$ for some $Z\neq 0$%
. Choose a maximal abelian subalgebra $\mathfrak{a}^{\prime }$ of $\mathfrak{%
p}$ that contains $Z$ and let 
\begin{equation*}
\mathfrak{a}_{Z}^{\prime }=\{H\in \mathfrak{a}^{\prime }\text{ }:\alpha (H)=0%
\text{ for all }\alpha \in \Phi _{Z}\},
\end{equation*}%
where we understand the root system $\Phi $ to be with respect to this
subalgebra $\mathfrak{a}^{\prime }$. For each $\alpha \in \Phi ^{+}$, choose
bases $\{E_{\alpha }^{(i)}:i\in I_{\alpha }\}$ for the restricted root
spaces $\mathfrak{g}_{\alpha }$.

Temporarily fix an index $j$. Since $Ad(k_{j})Z_{j}\in \mathfrak{p}$, we can
find $H\in \mathfrak{a}^{\prime }$ and coefficients $c_{\alpha }^{(i)}$
(depending on $j$) such that%
\begin{equation*}
Ad(k_{j})Z_{j}=H+\sum_{\alpha \in \Phi ^{+}}\sum_{i\in I_{\alpha }}c_{\alpha
}^{(i)}E_{\alpha }^{(i)-}.
\end{equation*}%
Now, $Z\in Ad(k_{j})n_{Z_{j}}$, hence there is some $Y_{j}\in n_{Z_{j}}$
such that $Z=Ad(k_{j})Y_{j}$. Thus%
\begin{equation*}
\lbrack
Z,Ad(k_{j})Z_{j}]=[Ad(k_{j})Y_{j},Ad(k_{j})Z_{j}]=Ad(k_{j})[Y_{j},Z_{j}]=0.
\end{equation*}%
But we also have 
\begin{equation*}
\lbrack Z,Ad(k_{j})Z_{j}]=[Z,H+\sum_{\alpha \in \Phi ^{+}}\sum_{i}c_{\alpha
}^{(i)}E_{\alpha }^{(i)-}]=\sum_{\alpha \in \Phi ^{+}}\sum_{i\in I_{\alpha
}}\alpha (Z)c_{\alpha }^{(i)}E_{\alpha }^{(i)-}.
\end{equation*}%
It follows that $c_{\alpha }^{(i)}=0$ for all $\alpha \in \Phi ^{+}$ such
that $\alpha (Z)\neq 0$, i.e., for all $\alpha \notin \Phi _{Z}$. Hence 
\begin{equation*}
Ad(k_{j})Z_{j}=H+\sum_{\alpha \in \Phi _{Z}^{+}}\sum_{i\in I_{\alpha
}}c_{\alpha }^{(i)}E_{\alpha }^{(i)-}.
\end{equation*}%
Pick $H^{\prime }\in \mathfrak{a}_{Z}^{\prime }$. Since $\alpha (H^{\prime
})=0$ for all $\alpha \in \Phi _{Z}$, we have 
\begin{equation*}
\lbrack H^{\prime },Ad(k_{j})Z_{j}]=[H^{\prime },H+\sum_{\alpha \in \Phi
_{Z}^{+}}\sum_{i\in I_{\alpha }}c_{\alpha }^{(i)}E_{\alpha }^{(i)-}]=0.
\end{equation*}%
Thus $\mathfrak{a}_{Z}^{\prime }\subseteq Ad(k_{j})n_{Z_{j}}$ for all $j$
and hence is contained in $\bigcap\limits_{j=1}^{m}Ad(k_{j})n_{Z_{j}}$. To
complete the proof, simply choose a maximally singular element in $\mathfrak{%
a}_{Z}^{\prime }$.
\end{proof}

\begin{remark}
Note that the proof actually shows that the intersection contains all the
elements of $\mathfrak{a}_{Z}^{\prime }$ whose set of annihilating roots is
co-rank one in $\mathfrak{a}^{\prime }$.
\end{remark}

There are only finitely many co-rank one root subsystems of the (original)
restricted root system $\Phi $, so we may choose a finite set $S\subseteq 
\mathfrak{a}$ such that $\{\Phi _{Z}:Z\in S\}$ is the complete set. For each 
$Z\in S$, consider the map $f_{Z}:O_{Z}\times K^{m}\rightarrow \mathfrak{k}%
^{m}$ defined by%
\begin{equation*}
f_{Z}(Z^{\prime },k_{1},...,k_{m})=\left( [Z^{\prime
},Ad(k_{1})Z_{1}],...,[Z^{\prime },Ad(k_{m})Z_{m}]\right) .
\end{equation*}%
Note that $f_{Z}(Z^{\prime },k_{1},...,k_{m})=0$ if and only if $Z^{\prime
}\in \bigcap\limits_{j=1}^{m}Ad(k_{j})n_{Z_{j}}$.

For $Z\in \mathfrak{a}$, set%
\begin{eqnarray*}
G_{Z} &=&\{g\in G:Ad(g)Z=Z\}\text{ and } \\
K_{Z} &=&\{k\in K:Ad(k)Z=Z\}.
\end{eqnarray*}%
The associated Lie algebras are given by:%
\begin{eqnarray*}
\mathfrak{g}_{Z} &=&\{X\in \mathfrak{g}:[X,Z]=0\} \\
\mathfrak{k}_{Z} &=&\{X\in \mathfrak{k}:[X,Z]=0\} \\
\mathfrak{p}_{Z} &=&\{X\in \mathfrak{p}:[X,Z]=0\}.
\end{eqnarray*}

Let $\left( G_{Z}\right) _{0}$ and $\left( K_{Z}\right) _{0}$ be the
connected components containing the identity of $G$. Their Lie algebras are
also $\mathfrak{g}_{Z}$ and $\mathfrak{k}_{Z}$, respectively.

\begin{lemma}
The pair $\left( G_{Z}\right) _{0}/\left( K_{Z}\right) _{0}$ is a symmetric
space whose rank is equal to that of the dimension of $\mathfrak{a}$.
Moreover, $\mathfrak{g}_{Z}$ $=\mathfrak{k}_{Z}\oplus \mathfrak{p}_{Z}$.
\end{lemma}

\begin{proof}
Let $h\in \left( G_{Z}\right) _{0}$ and pick $H\in \left( \mathfrak{g}%
_{Z}\right) _{0}$ such that $h=\exp H$. Standard facts imply 
\begin{eqnarray*}
Ad(\theta h)Z &=&Ad(\theta \exp H)Z=Ad(\exp (d\theta )_{e}H)Z \\
&=&\exp (ad((d\theta )_{e}H)Z=Z
\end{eqnarray*}%
since $[(d\theta )_{e}H,Z]=0$. Thus $\theta h\in \left( G_{Z}\right) _{0}$.

Since $Z\in \mathfrak{a}$, $[X,Z]=0$ for all $X\in \mathfrak{a}$. Thus $%
\mathfrak{a\subseteq p}_{Z}$ and hence must be a maximal abelian subalgebra.
\end{proof}

\begin{lemma}
The rank of $Df_{Z}$ at $(Z^{\prime },k_{1},...,k_{m})\in f_{Z}^{-1}(0)$ is
at least%
\begin{equation*}
\sum_{j=1}^{m}\min_{\sigma \in W}\dim \left( \mathcal{N}_{Z}\bigcap 
\mathcal{N}_{\sigma (Z_{j})}\right) .
\end{equation*}
\end{lemma}

\begin{proof}
Fix $(Z^{\prime },k_{1},...,k_{m})\in f_{Z}^{-1}(0)$. Consider the $j$'th
inclusion map: $K\rightarrow O_{Z}\times K^{m}$ given by $k\longmapsto
(Z^{\prime },k_{1},...,k_{j-1},k,k_{j+1},...,k_{m})$ and let $%
f_{Z}^{(j)}:K\rightarrow \mathfrak{k}^{m}$ be the composition of this
inclusion with $f_{Z}$.

As the derivative of $f_{Z}^{(j)}$ lies in the $j$'th coordinate of $%
\mathfrak{k}^{m}$, 
\begin{equation*}
rankDf_{Z}\geq \sum_{j=1}^{m}rankDf_{Z}^{(j)}\text{.}
\end{equation*}

Suppressing the unused components of the domain of $f_{Z}^{(j)}$, we can
write $f_{Z}^{(j)}(k)=[Z^{\prime },Ad(k)Z_{j}]$. We will compute the rank of 
$f_{Z}^{(j)}$ at $k_{j}$.

We claim that there is a $k\in K$ such that $Ad(k)Z^{\prime }=Z$ and $%
Ad(kk_{j})Z_{j}=\sigma (Z_{j})$ for some $\sigma \in W$. To see this, choose 
$h_{1}\in K$ such that $Ad(h_{1})Z^{\prime }=Z$ and consider $%
[Z,Ad(h_{1}k_{j})Z_{j}].$ As $Ad(h_{1}k_{j})Z_{j}\in \mathfrak{p}_{Z}$,
there is some $h_{2}\in (K_{Z})_{0}$ such that $Ad(h_{2})Ad(h_{1}k_{j})Z_{j}%
\in \mathfrak{a}$. Set $k=h_{2}h_{1}$. By construction, $Ad(k)Z^{\prime }=Z$%
. Since $kk_{j}\in K$ and $Ad(kk_{j})Z_{j}\in \mathfrak{a}$, the claim now
follows from the fact that if two elements, $Q,Q^{\prime }$, of $\mathfrak{a}
$ are $Ad(K)$ related, then there is an element $\sigma \in W$ with $\sigma
(Q)=Q^{\prime }$.

Since the rank of $f_{Z}^{(j)}$ is $Ad$-invariant, we can assume that $%
Z^{\prime }=Z$ and $Ad(k_{j})Z_{j}=\sigma (Z_{j})$.

For $X\in \mathfrak{k}$, we have 
\begin{equation*}
(Df_{Z}^{(j)})_{k_{j}}(X)=[Z,Ad(k_{j})[X,Z_{j}]]=[Z,[Ad(k_{j})X,\sigma
(Z_{j})]].
\end{equation*}%
Thus%
\begin{equation*}
\text{Im}(Df_{Z}^{(j)})_{k_{j}}=[Z,\mathcal{N}_{\sigma (Z_{j})}]=sp\left\{
X_{\alpha }^{+}:X_{\alpha }\in \mathfrak{g}_{\alpha },\alpha \in \Phi
_{Z}^{c}\bigcap \Phi _{\sigma (Z_{j})}^{c}\right\}
\end{equation*}%
and this has the same dimension as $\mathcal{N}_{Z}\bigcap \mathcal{N}%
_{\sigma (Z_{j})}$.
\end{proof}

\begin{proof}
\lbrack of Theorem] The proof of the theorem can now be completed as in \cite%
{Wr}. The hypothesis of the theorem implies that 
\begin{equation*}
\sum_{j=1}^{m}\min_{\sigma \in W}\dim \left( \mathcal{N}_{Z}\bigcap 
\mathcal{N}_{\sigma (Z_{j})}\right) >\dim \mathcal{N}_{Z}
\end{equation*}%
for all maximally singular elements $Z$, thus the rank of $f_{Z}$ at any
element of $f_{Z}^{-1}(0)$ is greater than the dimension of $O_{Z}$.
Consequently, $f_{Z}^{-1}(0)$ has dimension less than $K^{m}$. If we let $%
\pi _{Z}:O_{Z}\times K^{m}\rightarrow K^{m}$ be the projection, then $\pi
_{Z}(f_{Z}^{-1}(0)),$ and thus also $\bigcup\limits_{Z\in S}\pi
_{Z}(f_{Z}^{-1}(0))$, has measure zero in $K^{m}$. If $(k_{1},...,k_{m})%
\notin \bigcup\limits_{Z\in S}\pi _{Z}(f_{Z}^{-1}(0))$, then $%
\bigcap\limits_{j=1}^{m}Ad(k_{j})n_{Z_{j}}=\{0\}$ and hence $%
O_{Z_{1}}+\cdot \cdot \cdot +O_{Z_{m}}$ has non-empty interior.
\end{proof}

\subsection{Completion of the proof of sufficiency for symmetric spaces with
all one-dimensional, restricted root spaces}

In this subsection we will apply the previous results to prove that
eligible, non-exceptional $L$-tuples in symmetric spaces with (all)
one-dimensional, restricted root spaces are absolutely continuous. We begin
with two lemmas that will allow us to establish the base cases.

We call $Z\in \mathfrak{a}$ \textit{regular} if $\Phi _{Z}$ is empty.

\begin{lemma}
\label{reg}If $Z$ is a regular element and $Y\in \mathfrak{a}$ is non-zero,
then $\mu _{Z}\ast \mu _{Y}$ is absolutely continuous.
\end{lemma}

\begin{proof}
The proof given in \cite{GHgeneric} can be easily adapted or see \cite%
{GS2003}.
\end{proof}

\begin{lemma}
Any eligible, non-exceptional pair, $(X,Y),$ in the symmetric space $%
G_{n}/K_{n}$ where $G_{n}=$ $SO_{0}(n,n)$, $K_{n}=SO(n)\times SO(n)$ for $%
n=3,4,5,$ is absolutely continuous.
\end{lemma}

\begin{proof}
When $n=3$, the only eligible, non-exceptional pairs where neither $X$ nor $%
Y $ are regular are the pairs of type $(D_{2},SU(2))$ and $(SU(2),SU(2))$.
Since the annihilators of any element of type $SU(2)$ are contained in the
annihilators of an element of type $D_{2},$ it suffices to check that the
former pair is absolutely continuous. For this one can easily verify (\ref%
{WrC}).

The induction argument, Theorem \ref{indstep}, can then be called upon to
see that all eligible pairs $(X,Y)$ in $G_{n}/K_{n}$ with $n=4$ are
absolutely continuous, except those for which $X^{\prime }$ is type $SU(3)$
and $Y^{\prime }$ is either $SU(3))$ or $SU(2)$. These are the pairs $(X,Y)$
of types $(SU(4),SU(4))$, $(SU(4),SU(3))$, $(SU(4),D_{2}\times SU(2))$, or $%
(SU(4),SU(2)\times SU(2))$. Notice that all these are exceptional pairs,
except for the pairing $X$ of type $SU(4)$ and $Y$ of type $SU(2)\times
SU(2))$ with $\Phi _{Y}$ not Weyl conjugate to a subset of $\Phi _{X}$. For
the last pair we can easily check (\ref{WrC}) is satisfied$.$

The arguments are similar for $n=5$; it suffices to check the Wright
criterion for a pair of type $(SU(5),D_{2}\times SU(3))$ and this can be
done as in \cite[Lemma 6]{GHAbs}.
\end{proof}

\begin{theorem}
\label{Propdim1}Suppose $G/K$ is a symmetric space whose restricted root
spaces are all of dimension one. If $(X,Y)$ is an eligible, non-exceptional
pair of non-zero elements in $\mathfrak{a},$ then $(X,Y)$ is an absolutely
continuous pair.
\end{theorem}

\begin{proof}
This will be an induction argument based on the rank of the symmetric space.
The previous lemma establishes the result for the symmetric spaces with $%
G_{n}=SO_{0}(n,n)$ and $K_{n}=SO(n)\times SO(n)$ (restricted root space type 
$D_{n})$ for $n=3,4,5$. Hence for this family of symmetric spaces we can
start the induction argument with rank $n=6$, taking $n=5$ as the base case.
All the other symmetric spaces under consideration have restricted root
spaces of types $A_{n}$, $B_{n}$ or $C_{n}$. For all of these we can take $%
n=1$ as the base case and there the result trivially holds by Lemma \ref{reg}
since any non-zero element is regular.

We now assume inductively that the result holds for symmetric spaces of rank 
$n-1$ and proceed to consider the problem for rank $n$. Assume $(X,Y)$ is an
eligible, non-exceptional pair of non-zero elements in $\mathfrak{a}_{n}$.
From Lemma \ref{eligible} we know the reduced pair, $(X^{\prime },Y^{\prime
})$, is also eligible. If the restricted root system is type $B_{n}$ or $%
C_{n}$, then clearly $(X^{\prime },Y^{\prime })$ is not exceptional. If the
restricted root system is type $D_{n}$ (with $n\geq 6$) then $(X^{\prime
},Y^{\prime })$ can only be exceptional if $X^{\prime }$ is type $SU(n-1)$
and $Y^{\prime }$ is either that type or type $SU(n-2)$. But this happens
only if $X$ is of type $SU(n)$ and $Y$ is either type $SU(n)$ or $SU(n-1)$
which is not true as the pair $(X,Y)$ is not exceptional. Lastly, we remark
that in the case of type $A_{n}$, we can be sure the pair $(X^{\prime
},Y^{\prime })$ is not exceptional because if $X^{\prime },Y^{\prime }$ were
both of type $SU(n/2)\times SU(n/2)$ in the symmetric space with restricted
root system of type $A_{n-1}$, then $X,Y$ would both be type $%
SU(n/2+1)\times SU(n/2)$ and that's not an eligible pair in the original
symmetric space.

The induction hypothesis thus implies that $(X^{\prime },Y^{\prime })$ is an
absolutely continuous pair. Appealing to Theorem \ref{indstep}, we conclude
that the same is true for $(X,Y)$.
\end{proof}

We now turn to the problem of $L\geq 3$ where the results are new for all
types. Type $D_{n}$ is the most complicated because of the exceptional
cases. (The exceptional pairs present difficulties even for dealing with $%
L\geq 3$.) These problems were also addressed for the Lie algebra case in 
\cite{GHAbs} (see especially Lemmas 6, 7 in that paper), but in \cite{GHAbs}
some of the arguments relied upon $L^{2}$ density results for convolutions
of orbital measures and such results are generally unknown in the symmetric
space setting.

We begin the argument with several technical lemmas which will enable us to
address these complications.

\begin{lemma}
\label{base}Consider the symmetric space of Cartan type $DI$, $%
SO_{0}(n,n)/SO(n)\times SO(n)$ with $n\geq 3$. Let $X,Y\in \mathfrak{a}_{n}$
be of dominant $SU$ type and $Z\in \mathfrak{a}_{n}$ be non-zero.

\begin{enumerate}
\item When $n=4$, the triple $(X,Y,Z)$ is absolutely continuous if $X,Y,Z$
are all type $SU(4),$ but their annihilating root systems are not Weyl
conjugates.

\item Suppose $n\geq 4$. If $(X^{\prime },Y^{\prime },Z^{\prime })$ is an
absolutely continuous triple, then $(X,Y,Z)$ is also absolutely continuous.

\item If $n\geq 4$ and $Z$ is also of dominant $SU$ type, then $(X,Y,Z)$ is
an absolutely continuous triple, except if $n=4$ and all three of $X,Y,Z$
are type $SU(4)$ with Weyl conjugate sets of annihilating roots.

\item If $n\geq 4$ and $Z$ is not of dominant $SU$ type, then $(X,Y,Z)$ is
absolutely continuous.
\end{enumerate}
\end{lemma}

\begin{proof}
(1) We use the criterion of (\ref{WrC}) to prove this. The root subsystems
of rank 3 in $D_{4}$ are of type $D_{2}\times SU(2)$, $D_{3}$ and $SU(4)$.
The key points to observe are:

(i) the intersection of any positive root systems of type $SU(4)$ with one
of type $D_{J}\times SU(4-J)$ has cardinality at least $1$ if $J=2$ and at
least $3$ if $J=3$;

(ii) the intersection of any two Weyl conjugate positive root systems of
type $SU(4)$ has cardinality at least four; the intersection of any two
non-Weyl conjugate positive root systems of type $SU(4)$ has cardinality at
least six.

(2) We will use the notation of Theorem \ref{indstep}. Let $\Omega
=\{E_{\alpha }^{-}:\alpha =e_{1}\pm e_{j},2\leq j\leq n\}$. Without loss of
generality we can assume $\Omega _{X}\supseteq \{E_{\alpha }^{-}:\alpha
=e_{1}+e_{j},2\leq j\leq n\}$ and $\Omega _{Y}$ either contains the same set
again or $\Omega _{Y}\supseteq \{E_{\alpha }^{-},E_{e_{1}-e_{n}}^{-}:\alpha
=e_{1}+e_{j},2\leq j\leq n-1\}$. Let $k\in K_{n}$ be a Weyl conjugate that
changes the signs of $2,...,n-1$\ (and $n$ if necessary, to be an even sign
change). Then $\Omega _{X}\bigcup \Omega _{Y}$ contains all of $\Omega $,
except possibly $E_{e_{1}-e_{n}}^{-}$.

If $Z$ is not type $SU(n)$, applying a Weyl conjugate, if needed, we can
assume $\Omega _{Z}\supseteq \{E_{\alpha }^{-}:\alpha =e_{1}\pm e_{n}\}$. We
take $\Omega _{0}=\{E_{e_{1}+e_{n}}^{-}\}$ and $M=E_{e_{1}+e_{n}}^{+}$. By
assumption, $(X^{\prime },Y^{\prime },Z^{\prime })$ is absolutely
continuous, hence we can appeal to the general strategy, Prop. \ref{general
strategy} to see that $(X,Y,Z)$ is absolutely continuous.

If $Z$ is of type $SU(n)$, after applying a Weyl conjugate we can assume $%
\Omega _{Z}$ contains $E_{e_{1}-e_{n}}^{-}$ and $E_{\beta }^{-}$ for $\beta $
one of $e_{1}\pm e_{2}$. Take $\Omega _{0}=\{E_{\beta }^{-}\}$ and $%
M=E_{\beta }^{+}$, and appeal to the same proposition again.

(3)\ It is useful to note that if $\mu _{X}\ast \mu _{Y}$ is absolutely
continuous, then so is $\mu _{X}\ast \mu _{Y}\ast \mu _{Z}$ for any $Z$.

Assume, first, that $n\geq 5$. Since all pairs of dominant $SU$ type are
absolutely continuous except the pairs $(SU(n),SU(n))$ and $(SU(n),SU(n-1))$%
, and the annihilating root system of any element of type $SU(n-1)$ is
contained in one of type $SU(n)$, it suffices to check that the triple $%
(SU(n),SU(n),SU(n))$ is absolutely continuous. To do this, we will verify
the result holds for $n=5$ and then appeal to the induction argument
established in the previous part of this lemma.

The most efficient way to prove a triple $(X,Y,Z)$, each of type $SU(5)$
with $n=5$, is absolutely continuous is to establish that the reduced triple 
$(X^{\prime },Y^{\prime },Z^{\prime })$ is absolutely continuous and again
appeal to the previous part of the lemma. The reduced elements are each of
type $SU(4)$, but by applying an even number of sign changes as needed, to
the original triple, we can assume the three $SU(4)$ annihilating root
subsystems are not Weyl conjugate. Thus part (1) of the lemma establishes
the absolute continuity of $(X^{\prime },Y^{\prime },Z^{\prime })$.

When $n=4$, the pairs $(X,Y)$ that are each of dominant $SU$ type that are
not absolutely continuous are those where (without loss of generality) $X$
is type $SU(4)$ and $Y$ is type $SU(4),SU(3)$ or $SU(2)\times SU(2)$ where,
in the latter case, $\Phi _{Y}\subseteq \Phi _{X}$. Thus it will be enough
to check that the triples $(SU(4),SU(4),SU(3))$ and $(SU(4),SU(4),SU(2)%
\times SU(2))$ are absolutely continuous. The first follows from part (1)
since any root system of type $SU(3)$ can be viewed as a subset of either of
the two Weyl conjugacy classes of root systems of type $SU(4)$. \ The second
can be deduced from (\ref{WrC}).

(4) First assume $n=4$. The triples with precisely two terms that are
dominant $SU$ type, that cannot be seen to be absolutely continuous by
arguing that some pair in the triple is absolutely continuous, are the types 
$(SU(4),SU(4),D_{3})$, $(SU(4),SU(3),D_{3})$ and $(SU(4),SU(4),D_{2}\times
SU(2))$. Actually, it is enough to check the first and last triple in the
list, as the absolute continuity of the second follows from the first. These
we check by (\ref{WrC}).

Together with part (3), we have now shown that all triples in $D_{4},$ where
at least two elements are dominant $SU$ type, are absolutely continuous,
other than in the exceptional case where all three elements are type $SU(4)$
with Weyl conjugate sets of annihilating root systems.

Now assume $n\geq 5$. There is no loss in assuming $X$ and $Y$ are type $%
SU(n)$ and $Z$ is not dominant $SU$ type. We proceed by induction using the
comment of the previous paragraph to start the base case, $n=4$, noting that 
$X^{\prime }$ and $Y^{\prime }$, but not $Z^{\prime }$, are type $SU(n-1)$.
If $Z^{\prime }$ is not dominant $SU$ type, then by the induction hypothesis
the triple $(X^{\prime },Y^{\prime },Z^{\prime })$ is absolutely continuous
and we appeal to part (2) to conclude that $(X,Y,Z)$ is absolutely
continuous. If $Z^{\prime }$ is dominant $SU$ type, we appeal to (3) to see
that the reduced triple is absolutely continuous and then call again upon
(2).
\end{proof}

Here is a useful and immediate corollary.

\begin{corollary}
\label{2dom}Suppose $n\geq 5$, non-zero $X_{j}\in \mathfrak{a}_{n}$ for $%
j=1,...,L\geq 3$ and at least two $X_{j}$ are dominant $SU$ type. Then $%
(X_{1},...,X_{L})$ is an absolutely continuous tuple.
\end{corollary}

\begin{lemma}
\label{domswitch}Consider the symmetric space $SO_{0}(n,n)/SO(n)\times SO(n)$
with $n\geq 5$.

\begin{enumerate}
\item Suppose $X$ is of dominant $SU$ type, $Y$ is of dominant $D$ type and $%
Y^{\prime }$ is of dominant $SU$ type. Then the pair $(X,Y)$ is absolutely
continuous.

\item Suppose $X,Y$ are both of dominant $D$ type, but $X^{\prime }$, $%
Y^{\prime }$ are both of dominant $SU$ type. Then the pair $(X,Y)$ is
absolutely continuous.
\end{enumerate}
\end{lemma}

\begin{proof}
Suppose an element $Z$ is type $D_{J}\times SU(s_{1})\times \cdot \cdot
\cdot \times SU(s_{m})$ and is of dominant $D$ type, but $Z^{\prime }$ is
dominant $SU$ type. Then $2J>s_{1}\geq 2J-1$ and consequently, $J\leq
(n+1)/3 $. Using this fact it is easy to check that under either of the two
hypotheses, the pair $(X,Y)$ is eligible and not exceptional.
\end{proof}

\begin{theorem}
\label{Thmdim1}Suppose $G/K$ is a symmetric space whose restricted root
spaces all have dimension one. Let $L\geq 3$. If $(X_{1},X_{2},...,X_{L})$
is an eligible, non-exceptional $L$-tuple of non-zero elements in $\mathfrak{%
a},$ then $(X_{1},X_{2},...,X_{L})$ is absolutely continuous.
\end{theorem}

\begin{proof}
We will first give the argument when the restricted root system is type $%
D_{n}$; the other cases are easier, as we indicate below. The key idea is
again an induction argument that is similar to one used in \cite{GHAbs} and
builds upon the $L=2$ result.

The base cases $D_{3}$ and $D_{4}$ will be discussed at the conclusion of
the proof, so assume that the result is true for $n=3,4$ and that now $n\geq
5$. We will let 
\begin{equation*}
\Omega =\{E_{\alpha }^{-}:\alpha =e_{1}\pm e_{j},2\leq j\leq n\}.
\end{equation*}

By appealing to Cor. \ref{2dom} and Lemma \ref{domswitch}, we can assume
that for at least $L-1$ indices $j$, (say all but $j=L$) both $X_{j}$ and $%
X_{j}^{\prime }$ are dominant $D$ type for otherwise there is some pair, $%
(X_{k},X_{\ell }),$ which is already absolutely continuous. For these
\thinspace $L-1$ indices we have $S_{X_{j}^{\prime }}=S_{X_{j}}-2$ and from
this it is easy to see that the reduced tuple, $(X_{1}^{\prime
},....,X_{L}^{\prime }),$ is eligible. It is clearly not exceptional. By the
induction assumption, the reduced tuple is absolutely continuous.

For $j\neq L$, we have $\Omega _{X_{j}}=\{E_{e_{1}\pm e_{k}}^{-}:k>J_{j}\}$
where $2J_{j}=S_{X_{j}}$. By choosing suitable Weyl conjugates, $k_{i}\in K$%
, we can arrange for $\bigcup\limits_{i=1}^{L-1}Ad(k_{i})\left( \Omega
_{X_{i}}\right) $ to contain $\Omega _{Y}$ for some choice of $Y$ which is
of type $D_{m}$ for $m=n-(L-1)n+\sum_{i=1}^{L-1}J_{i}$ (or regular $Y$ if $%
m<2$); for more details on how to do this, we refer the reader to \cite%
{GHAbs}.

The eligibility assumption ensures that the pair $(Y,X_{L})$ is eligible and
not exceptional. The arguments given in the proof of Theorem \ref{indstep}
Case 1 or Case 3, depending on whether $X_{L}$ is dominant $D$ or dominant $%
SU$ type, can then be used to show that there is some $k$ $\in K_{n-1}$, $%
M\in \mathfrak{k}_{n}$ and $\Omega _{0}\subseteq \Omega _{X_{L}}$ such that 
\begin{eqnarray*}
sp\Omega  &=&sp\{Ad(k)(\Omega _{Y}),\Omega _{X_{L}}\backslash \Omega _{0}\}
\\
&=&sp\{Ad(kk_{i})(\Omega _{X_{i}}),\Omega _{X_{L}}\backslash \Omega
_{0}:i=1,....,L-1\};
\end{eqnarray*}%
$ad^{k}(M)$ maps $\mathcal{N}_{X_{L}}\backslash \Omega _{0}\rightarrow
sp\{\Omega ,\mathfrak{p}_{n-1}\}$ for all positive integers $k$; and the
span of the projection of $Ad(\exp sM)(\Omega _{0})$ onto the orthogonal
complement of $sp\{\mathfrak{p}_{n-1},\Omega \}$ in $\mathfrak{p}_{n}$ is a
surjection for all small $s>0.$

Calling upon the general strategy, Prop. \ref{general strategy}, with $k_{i}$
replaced there by $kk_{i},$ we deduce that $(X_{1},...,X_{L})$ is an
absolutely continuous tuple.

This completes the induction argument and we now turn to the base cases.

When the restricted root system is type $D_{3}$, we want to show all
four-tuples are absolutely continuous and all triples are absolutely
continuous, except when all three elements are of type $SU(3)$. This reduces
to checking that the triples $(SU(3),SU(3),D_{2})$, $(SU(3),D_{2},D_{2})$
and $(D_{2},D_{2},D_{2})$ are absolutely continuous, and that the four-tuple
consisting of all elements of type $SU(3)$ is absolutely continuous. These
are easily checked using the criteria (\ref{WrC}), noting that any two root
systems of type $SU(3)$ will intersect non-trivially.

For type $D_{4},$ we have already seen in Lemma \ref{base} that any triple
with at least two terms that are dominant $SU$ type is absolutely
continuous, other than the exceptional triple. Thus we only need to check
the triples with two or three terms that are dominant $D$ type. But in any
of these cases the reduced triple in type $D_{3}$ will be absolutely
continuous and we can use the induction argument given earlier in this
proof, taking $D_{3}$ as the base case. This finishes the argument for
restricted root systems of type $D_{n}$.

When the restricted root system is type $B_{n}$ or $C_{n}$ the base case
argument is trivial because when $n=1$ the convolution of any two non-zero
orbital measures is absolutely continuous. To begin the induction argument,
we note that if two or more $X_{i}$ are dominant $SU$ type, then that pair
is itself eligible and hence absolutely continuous. Similarly, if two or
more $X_{i}^{\prime }$ are dominant $SU$ type, it is easy to see that the
corresponding $X_{i}$ are an eligible pair and hence are absolutely
continuous. Thus, again we can assume $S_{X_{i}^{\prime }}=S_{X_{i}}-2$ for
all but at most one $i$ and that ensures the reduced tuple is eligible. Now
apply the induction argument as done for type $D_{n}$ above, (starting with $%
\Omega =\{E_{\alpha }^{-},E_{(2)e_{1}}^{-}:\alpha =e_{1}\pm e_{j},2\leq
j\leq n\}$) obtaining a $Y$ that is either type \thinspace $B_{m}$ (or $C_{m}
$) if $m=n-(L-1)n+\sum_{i=1}^{L-1}J_{i}\geq 2$ or $B_{1}\ $(or $C_{1})$
otherwise.

The argument is easier, still, for type $A_{n-1}$. If two or more $X_{i}$
have $S_{X_{i}^{\prime }}=S_{X_{i}}$, then these satisfy $S_{X_{i}}\leq n/2$
and this fact implies the reduced $L$-tuple is eligible. Otherwise, at most
one $X_{i}$ has $S_{X_{i}^{\prime }}=S_{X_{i}}$ and again we deduce that the
reduced tuple is eligible. Now apply the induction argument in a similar
manner.
\end{proof}

\subsection{Proof of sufficiency for symmetric spaces with higher
dimensional, restricted root spaces}

We prove two more technical lemmas before completing the proof of
sufficiency for symmetric spaces with higher dimensional, restricted root
spaces.

\begin{lemma}
\label{SU}In the symmetric spaces whose restricted root systems are type $%
B_{n}$, $C_{n}$ or $BC_{n}$, the pairs of type $(SU(n),SU(n))$ are
absolutely continuous.
\end{lemma}

\begin{proof}
This will be an induction argument on $n$, similar to the argument given in
Case 2 of Theorem \ref{indstep}. The base case, $n=1$, is trivial, so assume
the result holds for $n-1$, $n\geq 2$. We will write the proof for type $%
C_{n}$; only notational changes are needed for the other types. Let 
\begin{equation*}
\Omega =\{E_{e_{1}\pm e_{j}}^{(u)-},E_{2e_{1}}^{(v)-}:j=2,...,n;u,v\}\text{.}
\end{equation*}%
where $\{E_{e_{1}\pm e_{j}}^{(u)}:u\}$ is a basis for the restricted root
space $\mathfrak{g}_{e_{1}\pm e_{j}}$ and $\{E_{2e_{1}}^{(v)}:v\}$ a basis
for $\mathfrak{g}_{2e_{1}}$. Let $\Omega _{0}$ be any one of the vectors $%
E_{2e_{1}}^{(v)-}$. Applying a Weyl conjugate, if necessary, there is no
loss of generality in assuming 
\begin{equation*}
\Omega _{X}=\Omega _{Y}=\{E_{e_{1}\pm
e_{j}}^{(u)-},E_{2e_{1}}^{(v)-}:j=2,...,n;u,v\}.
\end{equation*}%
Taking the Weyl conjugate $k\in K_{n-1}$ that changes the sign of the
letters $2,...,n$, we have 
\begin{equation*}
\{Ad(k)(\Omega _{Y}),\text{ }\Omega _{X}\diagdown \Omega _{0}\}=\Omega \text{%
.}
\end{equation*}

As $X^{\prime },Y^{\prime }$ are both type $SU(n-1),$ they are an absolutely
continuous pair in $\mathfrak{p}_{n-1}$. Now take $M=E_{2e_{1}}^{+}$. Since
the complement of $sp\{\mathfrak{p}_{n-1},\Omega \}$ in $\mathfrak{p}_{n}$
is spanned by any one basis vector in $\mathfrak{a}_{n}\ominus \mathfrak{a}%
_{n-1}$, an application of the general strategy, Prop. \ref{general strategy}%
, completes the argument.
\end{proof}

\begin{lemma}
\label{C4}The pair $(C_{k-2}\times SU(2),SU(k))$ is absolutely continuous in
any symmetric space whose restricted root system is type $C_{k}$, \thinspace 
$k=3,4$.
\end{lemma}

\begin{proof}
We will give the proof for $C_{4}$ and leave $C_{3}$ as an exercise. Suppose
the multiplicities of the long roots are $m_{L}$ and the multiplicities of
the short roots are $m_{S}$. In terms of this notation, $\dim \Phi
=12m_{S}+4m_{L}$. The co-rank one root subsystems are types $%
C_{3},C_{2}\times SU(2)$, $C_{1}\times SU(3)$ and $SU(4)$. The chart below
summarizes the pertinent information. When we write $\min \dim (\Psi
\bigcap \Psi ^{\prime })$ we mean the minimal dimension of the span of $%
\{X_{\alpha }^{-}:\alpha \in \sigma (\Psi )\bigcap \Phi ^{\prime }\}$ where 
$\Phi ^{\prime }$ is any root subsystem of type $\Psi ^{\prime }$ and $%
\sigma $ is any Weyl conjugate.

\begin{equation*}
\begin{array}{ccccc}
\Psi & C_{3} & C_{2}\times SU(2) & C_{1}\times SU(3) & SU(4) \\ 
\dim \Psi & 6m_{S}+3m_{L} & 3m_{S}+2m_{L} & m_{L}+3m_{S} & 6m_{S} \\ 
\min \dim (\Psi \bigcap SU(4)) & 3m_{S} & m_{S} & m_{S} & 2m_{S} \\ 
\min \dim (\Psi \bigcap C_{2}\times SU(2)) & m_{L}+m_{S} & m_{S} & 0 & m_{S}%
\end{array}%
\end{equation*}%
With these facts it is easy to check that the criterion (\ref{WrC}) is
satisfied.
\end{proof}

\textbf{Completion of the Proof of Sufficiency:} Theorems \ref{Propdim1} and %
\ref{Thmdim1} establish the absolute continuity of all eligible,
non-exceptional $L$-tuples in symmetric spaces all of whose restricted root
spaces have dimension one. That means the sufficiency result is proven for
the Cartan classes $AI$, $CI,$ $BI$ (with $q=p+1$) and $DI$ (with $p=q\geq 3$%
).

Lemma \ref{embedding1} shows that Cartan class $AI$ embeds into $AII$ (of
the same rank) with the identity map. Since the eligible, non-exceptional
tuples in type $AII$ are the same as those in type $AI$, sufficiency follows
for $AII$ directly from Prop. \ref{embedding2} (the embedding proposition).

Similarly, $BI$ with $q=p+1$ embeds into $BDI$ with $q>p$ and also into
Cartan types $AIII$ and $CII$ with $q>p$, in all cases with the identity
map. There are no exceptional tuples in restricted root systems of type $%
B_{p}$ (the restricted root system of type $BDI$ when $q>p$) and the
eligible tuples are the same in all these cases, hence we again obtain the
result for those types from the embedding proposition.

The Cartan type $DI$ with $p=q\geq 3$ has restricted root system type $D_{p}$
and embeds into type $AIII$ ($q=p$) with restricted root system of type $%
C_{p}$. This in turn embeds into type $CII$ also with restricted root system
of type $C_{p}$. Both embeddings are given by the identity map. Thus any
eligible $L$-tuple in type $AIII$ or $CII$ with $p\geq 3$, that is not
identified with an exceptional tuple in type $DI,$ is absolutely continuous.
By Lemma \ref{SU}, the eligible pair of type $(SU(p),SU(p))$ in either $AIII$
or $CII$ of rank $p$ is an absolutely continuous pair, and this implies the
same conclusion for the pair $(SU(p),SU(p-1))$ and the triple $%
(SU(p),SU(p),SU(p))$ when $p=3,4$. By Lemma \ref{C4}, the pairs $%
(C_{p-2}\times SU(2),SU(p))$ for $p=3,4$, and hence also the pair $%
(SU(2)\times SU(2),SU(4))$ when $p=4$, are absolutely continuous in types $%
AIII$ and $CII$. This shows that the eligible, but exceptional tuples in $DI$
are absolutely continuous tuples in types $AIII$ and $CII,$ when $p=q\geq 3$.

It was noted in the appendix that for Cartan type $AIII$ with $p=q$ we can
assume $p\geq 3$ and for type $CII$ with $p=q$ we can assume $p\geq 2$. For $%
CII$ with $p=q=2$ all pairs are eligible and we can check absolute
continuity by verifying the criterion (\ref{WrC}). This is very easy as the
only non-regular elements are types $C_{1}$ or $SU(2)$, so it remains only
to check the pairs $(C_{1},C_{1})$ and $(C_{1},SU(2))$. We leave the details
for the reader.

Lastly, consider the symmetric $SO^{\ast }(2n)/U(n),$ the space of Cartan
type $DIII$ where, as noted in the appendix, $n\geq 6$ when $n$ is even and $%
n\geq 3$ when $n$ is odd.

In Lemma \ref{embedding1} we saw that the Cartan type $IV$ symmetric space $%
SO(n,\mathbb{C})/SO(n)$ embeds into $SO^{\ast }(2n)/U(n)$ with the canonical
map. The symmetric spaces $SO(n,\mathbb{C})/SO(n)$ are dual to the compact
Lie groups $SO(n)$, considered as symmetric spaces, and have root systems of
type $D_{n/2}$ when $n\geq 6$ is even, or $B_{[n/2]}$ when $n\geq 3$ is odd.
In \cite{GHAbs} it was shown that all eligible, non-exceptional tuples in
these settings were absolutely continuous. In type $B_{[n/2]}$ there are no
exceptional tuples, so all eligible tuples in Cartan type $DIII$ with $n$
odd are absolutely continuous. When $n$ is even, the restricted root system
of Cartan type $DIII$ is type $C_{n/2}$, so again images of the exceptional
tuples from type $D_{n/2}$ must be shown to be absolutely continuous. This
is done in the same manner as for $AIII$ above.

\section{Proof of Necessity}

Finally, to complete the proof of the characterization theorem, we turn to
proving the necessity of eligibility and non-exceptionality.

\subsection{Eligibility is necessary}

\begin{proposition}
If $(X_{1},...,X_{L})\in \mathfrak{a}^{L}$ is not eligible, then $\mu
_{X_{1}}\ast \cdot \cdot \cdot \ast \mu _{X_{L}}$ is a singular measure.
\end{proposition}

\begin{proof}
We will prove necessity using properties of the underlying symmetric spaces,
but the core idea is elementary linear algebra.

Case: Cartan type $AI$ and $AII$ of rank $n-1$. Restricted root space is
type $A_{n-1}$.

Here $\mathfrak{g}=sl(n,F)$, the $n\times n$ matrices over $F$ with the real
part of their trace $0,$ where $F$ is $\mathbb{R}$ for type $AI$ and $F$ is
the Quaternions for type $AII$. The space $\mathfrak{p}$ consists of the
Hermitian members of $\mathfrak{g}$ and $\mathfrak{a}$ can be taken to be
the real diagonal matrices in $\mathfrak{p}$ (see \cite[p.371]{Kn}).

For $X\in \mathfrak{a}$, $S_{X}$ is the dimension of the largest eigenspace
of matrix $X$.

Given non-eligible $L$-tuple, $(X_{1},...,X_{L})\in \mathfrak{a}^{L}$, let $%
\alpha _{j}$ be the eigenvalue of $X_{j}$ (viewed as an element of $sl(n,F)$%
) with greatest multiplicity. Let $k_{j}\in K$ and denote by $V_{j}$ the
eigenspace of $Ad(k_{j})X_{j}$ corresponding to $\alpha _{j}$. Then 
\begin{equation*}
\dim \bigcap\limits_{j=1}^{L}V_{j}\geq \sum_{j=1}^{L}\dim
V_{j}-n(L-1)=\sum_{j=1}^{L}S_{X_{j}}-n(L-1)\geq 1.
\end{equation*}%
Now, for any $v\in \bigcap\limits_{j=1}^{L}V_{j}$ we have $%
(Ad(k_{j})X_{j})v=\alpha _{j}v$, hence the matrix $%
\sum_{j=1}^{L}Ad(k_{j})X_{j}$ has eigenvalue $\sum \alpha _{j}$. This proves
every element of $\sum_{j=1}^{L}O_{X_{j}}$ has eigenvalue $\sum \alpha _{j}$
and that implies$\sum_{j=1}^{L}O_{X_{j}}$ must have empty interior. Prop. %
\ref{keyprop} tells us $\mu _{X_{1}}\ast \cdot \cdot \cdot \ast \mu _{X_{L}}$
is a singular measure.

Case: Cartan type $CI$ of rank $n$. Restricted root space is type $C_{n}.$

In this symmetric space $\mathfrak{p}$ is the space of $2n\times 2n$
matrices of the form $\left[ 
\begin{array}{cc}
Z_{1} & Z_{2} \\ 
Z_{2} & -Z_{1}%
\end{array}%
\right] ,$ where $Z_{1},Z_{2}$ are $n\times n$ real matrices with $Z_{1}$
symmetric and $Z_{2}$ skew-symmetric. Take for $\mathfrak{a}$ the diagonal
matrices in $\mathfrak{p}$ (see \cite[p.454]{He}). We identify the diagonal
matrix $diag(b_{1},....,b_{n},-b_{1},...,-b_{n})$ with the $n$-vector $%
(b_{1},...,b_{n})$.

If $X\in \mathfrak{a}$ is type $C_{J}\times SU(s_{1})\times \cdot \cdot
\cdot \times SU(s_{t})$, then $0$ is an eigenvalue of $X$ with multiplicity $%
2J$ and $X$ has pairs of non-zero eigenvalues with multiplicity $s_{j}$. It
follows that the dimension of the largest eigenspace of $Ad(k)X$ is $S_{X}$
for any $k\in K$.

As above, we argue that if $(X_{1},...,X_{L})$ is not eligible, then every
element of $\sum_{j=1}^{L}O_{X_{j}}$ has a common eigenvalue and hence the
sum must have empty interior.

Case: Cartan types $BI$ and $DI$ - Symmetric space $SO_{0}(p,q)/SO(p)\times
SO(q)$ with $q\geq p$. Restricted root space is type $C_{p}$ (if $q>p)$ or
type $D_{p}$ if $q=p.$

Here $\mathfrak{p}$ consists of the $(p+q)\times (p+q)$ matrices $\left[ 
\begin{array}{cc}
0 & Z \\ 
Z^{t} & 0%
\end{array}%
\right] $ where $Z$ is a real $p\times q$ matrix. The space $\mathfrak{a}$
can be taken to be the set of matrices of the form 
\begin{equation}
X=\left[ 
\begin{array}{ccc}
0_{p\times p} & H & 0_{p\times (q-p)} \\ 
H & 0_{p\times p} & 0_{p\times (q-p)} \\ 
0_{(q-p)\times p} & 0_{(q-p)\times p} & 0_{(q-p)\times (q-p)}%
\end{array}%
\right]  \label{BTorus}
\end{equation}%
where $H$ is a real diagonal $p\times p$ matrix (see \cite{GSColloq}). One
can see from this description that a subset of $\mathfrak{p}$ in which every
element has $0$ as an eigenvalue with multiplicity greater than $q-p,$ or
contains only elements which have a common non-zero eigenvalue, cannot be
open.

We identify the matrix $X\in \mathfrak{a}$ with the $p$-vector whose entries
are the diagonal entries of $H$,%
\begin{equation}
X=(\underbrace{0,....,0}_{J},\underbrace{a_{1},...,a_{1}}_{s_{1}},...%
\underbrace{,a_{m},...,a_{m}}_{s_{m}}).  \label{nvector}
\end{equation}%
Then $0$ is an eigenvalue of $X$ with multiplicity $2J+q-p$ and each $\pm
ia_{j}$ is an eigenvalue of multiplicity $s_{j}$.

Suppose $(X_{1},...,X_{L})$ is not eligible and $k_{j}\in K$. If all $X_{j}$
are dominant $C$ or $D$ type and $V_{j}$ is the eigenspace of $%
Ad(k_{j})X_{j} $ corresponding to the eigenvalue $0$, then%
\begin{equation*}
\sum_{j=1}^{L}\dim V_{j}=\sum_{j=1}^{L}S_{X_{j}}+L(q-p)\geq (L-1)2p+1+L(q-p).
\end{equation*}%
Thus%
\begin{equation*}
\dim \bigcap\limits_{j=1}^{L}V_{j}\geq (L-1)2p+1+L(q-p)-(L-1)(p+q)\geq q-p+1.
\end{equation*}%
Consequently, $0$ is an eigenvalue of every element of $%
\sum_{j=1}^{L}O_{X_{j}}$ with multiplicity greater than $q-p$ and hence this
sum must have empty interior.

If, instead, one $X_{j}$ is dominant $SU$ type, then a similar argument
shows every element of $\sum_{j=1}^{L}O_{X_{j}}$ has a non-zero eigenvalue
with multiplicity at least one. Again it follows that $%
\sum_{j=1}^{L}O_{X_{j}}$ has empty interior.

If two or more $X_{j}$ are dominant $SU$ type, then $(X_{1},...,X_{L})$ is
eligible so we do not need to consider this case.

Case: Cartan types $AIII$ and $CII$

The arguments are the same as for $BDI$ as the only difference is that $%
\mathfrak{p}$ consists of complex or quaterion valued matrices.

Case: Cartan type $DIII$ with rank $m=[n/2]$. Restricted root space is $%
C_{m} $ or $BC_{m}$ depending on whether $n$ is even or odd.

Here $\mathfrak{p}$ consists of the purely imaginary, trace zero matrices of
the form $\left[ 
\begin{array}{cc}
Z_{1} & Z_{2} \\ 
Z_{2} & -Z_{1}%
\end{array}%
\right] ,$ where $Z_{1},Z_{2}$ are $n\times n$ symmetric matrices$.$ The
space $\mathfrak{a}$ consists of the matrices in $\mathfrak{p}$ of the form $%
X=\left[ 
\begin{array}{cc}
H & 0 \\ 
0 & -H%
\end{array}%
\right] $ where $H$ is block diagonal with $m$ $2\times 2$ blocks $\left[ 
\begin{array}{cc}
0 & b_{j} \\ 
b_{j} & 0%
\end{array}%
\right] $ when $n$ is even and and an additional entry of $0$ in the $(n,n)$
position if $n$ is odd (see \cite[p.454-5]{He}). We identify $X$ with the $m$%
-vector $(b_{1},...,b_{m})$.

If $X$ is type $(B)C_{J}\times SU(s_{1})\times \cdot \cdot \cdot \times
SU(s_{t})$ (depending on whether $n$ is even or odd), then $0$ is an
eigenvalue with eigenspace of dimension $4J$ if $n$ is even and $4J+2$ if $n$
is odd. The non-zero eigenvalues have multiplicities $2s_{j}$.

If $n$ is even, the dimension of the largest eigenspace is $2S_{X}$. As the
matrices in $\mathfrak{p}$ are size $4m\times 4m,$ the argument for the
necessity of eligibility is similar to type $CI$.

If $n$ is odd, then we require a slight variant on the argument. If all $%
X_{j}$ are dominant $BC$ type, then each has $0$ as its eigenvalue of
greatest multiplicity $2S_{X_{j}}+2$. If $V_{j}$ is the corresponding
eigenspace of $Ad(k_{j})X_{j}$, then 
\begin{eqnarray*}
\dim \bigcap_{j=1}^{L}V_{j} &\geq &\sum_{j=1}^{L}(2S_{X_{j}}+2)-(L-1)2n \\
&\geq &2((L-1)2m+1)+2L-2n(L-1)\geq 4.
\end{eqnarray*}%
Since the generic element of $\mathfrak{p}$ has $0$ as an eigenvalue with
multiplicity $2$, this implies $\sum_{j=1}^{L}O_{X_{j}}$ has empty interior.

If precisely one $X_{j}$ is of dominant $SU$ type, then as above we argue
that each element of $\sum_{j=1}^{L}O_{X_{j}}$ has a common non-zero
eigenvalue with multiplicity at least one. If two or more $X_{j}$ are
dominant $SU$ type, then $(X_{1},...,X_{L})$ is eligible.
\end{proof}

\subsection{Exceptional tuples are not absolutely continuous}

\begin{proposition}
If $(X_{1},...,X_{L})\in \mathfrak{a}_{n}^{L}$ is exceptional, then $\mu
_{X_{1}}\ast \cdot \cdot \cdot \ast \mu _{X_{L}}$ is a singular measure.
\end{proposition}

\begin{proof}
Different arguments will be needed for the different exceptional tuples.

1. Cartan type $DI$ of rank $n$

Case: $X$ is type $SU(n)$, $Y$ is either type $SU(n)$ or $SU(n-1)$;
restricted root system of type $D_{n}$.

It suffices to check the family of pairs $(SU(n),SU(n-1))$ is singular. One
can see that $\dim \mathcal{N}_{X}=\binom{n}{2}$ and $\dim \mathcal{N}_{Y}=%
\binom{n-1}{2}+2(n-1)$. Since $\dim \mathfrak{p}_{n}=n^{2},$ it is
impossible for $Ad(k_{1})\mathcal{N}_{X}+Ad(k_{2})\mathcal{N}_{Y}$ to equal $%
\mathfrak{p}_{n}$ for any choice of $k_{1},k_{2}$ and hence the pair is
singular.

Case: Other exceptional tuples in $D_{4}$.

Consider the isomorphism $\pi $ that identifies a root system of type $D_{3}$
with one of type $A_{3}$. With this identification, the exceptional tuples
are all identified with tuples that fail to be eligible in $D_{4}$ and hence
are not absolutely continuous. For example, a triple of Weyl conjugate root
subsystems of type $SU(4)$ is identified with a triple of root subsystems of
type $D_{3}$ in $D_{4}$ and such a triple is not eligible. The isomorphism $%
\pi $ lifts to an isomorphism of the symmetric space that preserves $%
\mathfrak{p}_{4}$ and $\mathfrak{k}_{4}$ and hence the failure of the
absolute continuity of the original tuples follows.

We remark that a different argument was given in \cite{GSArxiv} for the
exceptional pairs in type $D_{4}$.

Case: $X,Y,Z$ all type $SU(3)$ in $D_{3}$.

If such a triple was absolutely continuous, the induction argument, Lemma %
\ref{base}(2), would imply any triple of elements of type $SU(4)$ in $D_{4}$
would be absolutely continuous.

2. Cartan type $AI$ or $AII$ of rank $n-1$

Case: $X,Y$ both of type $SU(n/2)\times SU(n/2)$; restricted root system of
type $A_{n-1}$.

These pairs were proven to be singular in \cite{GSLie}. We note that for the
Cartan type $AI$, a dimension argument, as was given in the first case
above, would also establish the failure of absolute continuity.
\end{proof}

\section{Appendix}

In the chart below we list the irreducible, Riemannian globally symmetric
spaces of Type III. For each, we give the non-compact group $G$, the compact
subgroup $K$, its Cartan class, the Lie type of its restricted root system
and the dimensions of the restricted root spaces $\mathfrak{g}_{\alpha }$.
The rank is the subscript on the label of the restricted root system.

These details can be found in \cite[ch.X]{He}, \cite[VI.4]{Kn}, \cite[p.219]%
{Bu} and \cite[p.72]{CM}

\frame{$%
\begin{array}{cccccc}
\begin{array}{c}
\text{Cartan} \\ 
\text{class}%
\end{array}
& G/K & 
\begin{array}{c}
\text{Restricted} \\ 
\text{root system}%
\end{array}
& 
\begin{array}{c}
\dim \mathfrak{g}_{\alpha }\text{ for } \\ 
a=e_{i}\pm e_{j}%
\end{array}
& \alpha =e_{i} & \alpha =2e_{i} \\ 
AI & SL(n,\mathbb{R)}/SO(n),n\geq 2 & A_{n-1} & 1 & - & - \\ 
AII & SL(n,\mathbb{H)}/SO(n),n\geq 2 & A_{n-1} & 4 & - & - \\ 
AIII & 
\begin{array}{c}
SU(p,q)/SU(p)\times SU(q), \\ 
p=q\geq 3,q>p\geq 1%
\end{array}
& 
\begin{array}{c}
C_{p}\text{(if }p=q\text{)} \\ 
BC_{p}\text{(if }q>p\text{)}%
\end{array}
& 2 & 2(q-p) & 1 \\ 
CI & Sp(n,\mathbb{R)}/SU(n),n\geq 1 & C_{n} & 1 & 0 & 1 \\ 
CII & 
\begin{array}{c}
Sp(p,q)/Sp(p)\times Sp(q), \\ 
p=q\geq 2,q>p\geq 1%
\end{array}
& 
\begin{array}{c}
C_{p}\text{(if }p=q\text{)} \\ 
BC_{p}\text{(if }q>p\text{)}%
\end{array}
& 4 & 3 & 4(q-p) \\ 
DIII\text{(even)} & SO^{\ast }(2n)/U(n),n\geq 6 & C_{n/2} & 4 & 0 & 1 \\ 
DIII\text{(odd)} & SO^{\ast }(2n)/U(n),n\geq 3 & BC_{[n/2]} & 4 & 1 & 4 \\ 
\begin{array}{c}
BI\text{(}p+q\text{ odd)} \\ 
DI\text{(}p+q\text{ even)}%
\end{array}
& 
\begin{array}{c}
SO_{0}(p,q)/SO(p)\times SO(q), \\ 
q>p\geq 1%
\end{array}
& B_{p} & 1 & q-p & 0 \\ 
DI & 
\begin{array}{c}
SO_{0}(p,p)/SO(p)\times SO(p), \\ 
p\geq 3%
\end{array}
& D_{p} & 1 & 0 & 0%
\end{array}%
$}

We have omitted some from the list as they are isomorphic to others.

\begin{itemize}
\item $AIII$ with $p=q=1$ is isomorphic to $AI$ with $n=2$

\item $AIII$ with $p=q=2$ is isomorphic to $BI$ with $q=4,p=2$

\item $CII$ with $p=q=1$ is isomorphic to $BI$ with $q=4,p=1$

\item $DIII$ with $n=4$ is isomorphic to $DI$ with $q=6,p=2$
\end{itemize}

Type $DI$ with $p=2=q$ and type $DIII$ with $n=2$ are not irreducible.

We also describe below the restricted root systems of types $A_{n},B_{n}$, $%
C_{n},D_{n}$ and $BC_{n}$.%
\begin{equation*}
\frame{$%
\begin{array}{cc}
\begin{array}{c}
\text{Root system} \\ 
\text{type}%
\end{array}
& \text{Restricted root system }\Phi _{n}^{+} \\ 
A_{n} & \{e_{i}-e_{j}:1\leq i<j\leq n+1\} \\ 
B_{n} & \{e_{i},e_{i}\pm e_{j}:1\leq i\neq j\leq n\} \\ 
C_{n} & \{2e_{i},e_{i}\pm e_{j}:1\leq i\neq j\leq n\} \\ 
D_{n} & \{e_{i}\pm e_{j}:1\leq i\neq j\leq n\} \\ 
BC_{n} & \{e_{i},2e_{i},e_{i}\pm e_{j}:1\leq i\neq j\leq n\}%
\end{array}%
$}
\end{equation*}

\end{document}